\documentclass[a4paper,11pt,reqno]{amsart}

\usepackage{graphicx}
\usepackage{latexsym}
\usepackage{epsfig}
\usepackage{amssymb, amsmath, amsthm}
\usepackage{subfigure}
\usepackage{color}
\usepackage{enumitem}
\usepackage{mathrsfs} 
\usepackage{bbold}
\usepackage{caption}
\usepackage[cal=dutchcal]{mathalfa}
\usepackage[dvipdfm,a4paper,left=30mm,right=30mm,top=35mm,bottom=35mm]{geometry}

\newtheorem{thm}{Theorem}[section]
\newtheorem{prop}[thm]{Proposition}
\newtheorem{lem}[thm]{Lemma}

\newtheorem{rem}[thm]{Remark}

\newcommand{\R}{\mathbb{R}}
\newcommand{\Ueps}{U^{\varepsilon}}

\newcommand{\Weps}{W^{\varepsilon}}

\newcommand{\Phixi}{\Phi_{\xi}}

\newenvironment{rlcases}
  {\left \lbrace \begin{aligned}}
  {\end{aligned}\right\rbrace}

\DeclareMathOperator\sgn{sign}

\DeclareMathOperator\erf{erf}

\DeclareMathOperator\e{e}

\title{Parabolic free boundary price formation models  under market size fluctuations}

\author{Peter Markowich} 
\address{4700 King Abdullah University of Science and Technology, Thuwal 23955-6900, Kingdom of Saudi Arabia} 
\author{Josef Teichmann}
\address{Department of Mathematics, ETH Z\"urich, R\"amistr. 101, 8092 Z\"urich, Switzerland} 
\author{Marie-Therese Wolfram} 
\address{Mathematics Institute, University of Warwick, Coventry CV4 7TL, UK and Radon Institute for Computational and Applied Mathematics, Austrian Academy of Sciences, Altenbergerstr. 69, 4040 Linz, Austria. m.wolfram@warwick.ac.uk}

\begin{document}
\maketitle

\begin{abstract}
In this paper we propose an extension of the Lasry-Lions price formation model which includes fluctuations of the numbers of buyers and vendors. We analyze the model in the case of deterministic and stochastic market size fluctuations and present results on the long time asymptotic behavior and numerical evidence and conjectures on periodic, almost periodic and stochastic fluctuations. The numerical simulations  extend the theoretical statements and give further insights into price formation dynamics.
\end{abstract}



\pagestyle{myheadings}
\thispagestyle{plain}
\markboth{P.A. MARKOWICH, J. TEICHMANN AND M.T. WOLFRAM}{FB PRICE FORMATION MODELS WITH MARKET SIZE FLUCTUATIONS}

\section{Introduction}

This paper studies the impact of buyer and vendor number fluctuations on the price dynamics in an economic market in which a single good is traded. 
In general the formation of the price is a complex process, which results from the interplay of various additional factors such as trading behavior and rules, transaction costs, etc. Understanding these complicated interactions is a central research question in the  field of market microstructure, which aims to explain ``the process and outcome of exchanging assets under a specific set of rules'' (see for example O'Hara \cite{OHara}). \\
\noindent In 2007 Lasry and Lions introduced a minimalistic mean-field model (a one-dimensional parabolic free-boundary partial differential equation) to describe the price, which enters as the free boundary, of a single good traded between a large group of vendors and a large group of buyers. \\
\noindent Here we study a dynamic trading model, which is motivated by the price formation model proposed by Lasry and Lions. The original model considers a large group of buyers and vendors trading a single good. Each buyer or vendor acts continuously in time with his/her pre-trade price. If no trade takes place, the price diffuses instantaneously into its immediate neighborhood. When (through the diffusion around the pre-trade price) a buyer and vendor transact at this price, trading takes place. The buyer becomes a vendor, at a price which is the trading price plus the transaction cost, and the vendor becomes a buyer at a price which is the trading price minus the transaction cost. The total number of buyers as well as the total number of vendors is conserved in time. In reality however, buyers may not want to resell immediately or sellers may decide to wait with their next purchase. As a result temporal fluctuations in the number of buyers and vendors arise naturally and should be therefore included in the model. It is natural to expect that these fluctuations have independent increments and depend linearly on the total number of buyers and vendors.\\
\noindent There is a plethora of stochastic models, often based on queuing theory, for price formation (see for instance Lachapelle et al. \cite{laclaslehlio:13} and the references therein), which deal in different degrees of accuracy with the many known phenomena of price formation on a microlevel. The model, which we present here, fulfills this requirement only in a very schematic conceptual sense. It stands out, however, due to its analytic tractability, a feature badly missing in most of the other models. We believe that the stochastic version of the Lasry and Lions model can serve as a building block for other price formation models with more realistic description of the relaxation and fluctuation mechanisms. 
\\
The Lasry and Lions model has been studied in a series of papers (see \cite{CGGK2009, MMPW2009, GG2009, CMP2011, CMW2011}). Burger et al.~ \cite{BCMW2013, BCMW2014} identified the Lasry and Lions model as the asymptotic limit of a kinetic model, in which trading events between buyers and vendors are described by  collisions, giving Boltzmann-type equations for the buyer and vendor densities. The Lasry and Lions model corresponds to the high frequency trading regime, that is markets with  high volumes of trade and short holding periods. The short holding periods justify the mass conservation property of both models. Kinetic models were also proposed for various applications in finance, such as wealth and income distribution by D\"uring and Toscani, Toscani, Brugna and Demichelis as well as Pareschi and Toscani in \cite{DT2007, TBD2013,PT2013, PT2014}, knowledge growth in a society by Burger, Lorz and Wolfram  \cite{BLM2015} or opinion formation by Toscani \cite{T2006}. Degond, Liu and Ringhofer ~\cite{DLR2014, DLR2014-2} also proposed a kinetic mean-field approach to describe wealth evolution.\\
\noindent Here we generalize the Lasry and Lions model to account for fluctuations in the number of buyers and vendors. The fluctuation are modeled either deterministically or stochastically, and we aim to study their influence on the dynamics of the price. We discuss the long time behavior in either case as well as confirm and extend the analytic findings by several numerical simulations.\\

\noindent The paper is organized as follows: The parabolic price formation model of Lasry and Lions is introduced in Section \ref{s:ll}.
In Section \ref{s:deterministic} we present our generalization, which includes deterministic fluctuations in the masses of buyers and vendors. 
We discuss the large time asymptotics and periodic fluctuations in Section \ref{s:asymptotics}. In Section \ref{s:numsim} we illustrate the price dynamics with numerical simulations. In Section \ref{s:stochastic} we conclude by studying the price formation process in the case of stochastic mass fluctuations.

\section{The Lasry and Lions price formation model}\label{s:ll}

\noindent Lasry and Lions \cite{LL2007} consider a market in which a large group of buyers and a large group of vendors trade a certain good.  If a buyer and a seller agree on a price, a transaction takes place and the buyer immediately becomes a seller and vice versa. A transaction fee $a$ has to be paid by both parties. Buyers and vendors are modeled by densities, which describe the (absolute) number of buyers (vendors) over the price variable $x$. Price formation is modeled by buying at price $ x=p(t) $, paying -- due to a transaction cost $a$ -- the price $ p(t) + a $ and adding to the vendor distribution at $ p(t) + a $. The same dynamics hold in the case of selling (again with the transaction cost $a$). Non trivial dynamics are introduced by allowing for diffusive changes of the buyer and seller distributions between two events of price formation. The numbers of buyers and vendors remain unchanged at each time of the price formation process.\\
The distribution of the agents is described in terms of a signed density function $f$. In particular, the distribution of buyers over the price $x \in \R$ is the positive part $f$, that is, $f^+ = \max(f,0)$, while the negative part $f^- = -\min(f,0)$ is the distribution of vendors again over the price $x \in \R$. The distribution of buyers is supported to the left of the formed price $ p(t) $, whereas the distribution of vendors is supported to the right of $p(t)$ reflecting the old principle ``buy low and sell high''. The number of transactions at time $t$ is given by $\lambda = \lambda(t)=- k f_x(p(t),t)$. The impact of the trading events take place at the agreed price shifted by the transaction cost $a > 0$. The diffusive changes in the buyer and vendor distributions are modeled by second order derivatives with constant diffusivity $k = \frac{\sigma^2}{2}$. \\ 
\noindent All the above are expressed by the following free boundary problem: 
\begin{subequations}\label{e:lasrylions}
\begin{align}
&f_t = k f_{xx} + \lambda(t) (\delta(x-p(t)+a)-\delta(x-p(t)-a))  \text{ in } \mathbb{R} \times \mathbb{R}_+\label{e:llheat}\\
&\lambda(t) = - k f_x(p(t),t) \text{ and } f(p(t),t) = 0 \text{ in } \mathbb{R}_+\\
&f(\cdot,0) = f_I \text{ and } p(0) = 0.
\end{align}
\end{subequations}
Note that, without loss of generality, $p(0) = 0$. since it can always be achieved by a translation. We assume that the initial distribution of buyers and vendors $ f_I$ satisfies
\begin{align}\label{a:finit}
\begin{split}
&f_I \in L^{\infty}(\mathbb{R}) \cap L^1(\mathbb{R}),~xf_I \in L^1(\R) \\
&f_I \geq 0 \text{ in } (-\infty,0), ~ f_I(0) = 0 \text{ and } f_I(x) \leq 0 \text{ in } (0,+\infty).
\end{split}\tag{A1}
\end{align}
\noindent In the original model \eqref{e:lasrylions} the numbers of buyers $M^l$ and vendors $M^r$ are conserved in time and are given by
\begin{align} \label{e:masses}
M^l = \int_{\mathbb{R}} f^+_I(x)\, dx \text{ and } M^r = \int_{\mathbb{R}} f^-_I(x)\, dx.
\end{align}
The ratio of these ``masses'' determines the large time asymptotic behavior of the price, that is, if $M^l \neq M^r$, then (cf. \cite{CMP2011})
\begin{align}\label{e:asymp}
  p(t) \sim~ \sqrt{t} q_\infty \text{ with } \erf(q_\infty) =  \sqrt{4k} \frac{M^l - M^r}{M^r+M^l},
\end{align}
where we use the usual definition for the error function $\erf(z) := \frac{2}{\sqrt{\pi}}\int_0^z e^{-r^2} dr$.\\
As a result the price increases asymptotically like $\sqrt{t}$ if $M_l > M_r$ (number of buyers is larger than number of vendors) and decreases asymptotically like $-\sqrt{t}$ if $M_r > M_l$ (number of buyers smaller than the number of vendors). When $M_r = M_l$, the price stabilizes to a constant.\\

\noindent A very useful tool for the analysis and the numerics of the Lasry and Lions model is the following nonlinear transformation (introduced in \cite{CMP2011}), which changes \eqref{e:llheat} to the heat equation.
Define
\begin{align}\label{e:trans}
F(x,t) := \begin{cases}
 \hspace*{-1em}&\phantom{-}\sum_{n=0}^\infty f^+(x+na,t) \quad \text{if } x < p(t),\\ 
 \hspace*{-1em}&-\sum_{n =0}^\infty f^-(x-na,t) \quad \text{if } x > p(t).
\end{cases}
\end{align}
It follows that $F = F(x,t)$ satisfies the heat equation
\begin{align}\label{e:Fheat}
F_t = k F_{xx} \quad \text{ in } \mathbb{R}\times \mathbb{R}_+,
\end{align} 
with the (transformed) initial datum
\begin{align}\label{e:transfi}
F_I(x) = \begin{cases}
 \hspace*{-1em}&\phantom{-}\sum_{n=0}^\infty f_I^+(x+na) \quad \text{if } x < 0,\\
 \hspace*{-1em}&-\sum_{n=0}^\infty f_I^-(x-na) \quad \text{if } x > 0.
\end{cases}
\end{align}
Conversely, if $F$ is a solution to \eqref{e:Fheat}, then, the density $f$ in \eqref{e:lasrylions}, is recovered by 
\begin{align}\label{e:invF}
f(x,t) = F(x,t) - F^+(x+a,t) + F^-(x-a,t).
\end{align}
Note that to obtain \eqref{e:invF} it is necessary to ensure that the back-transformed function $f$ is positive to the left of the free boundary and negative to the right of it. This follows easily from the
structure of the transformed initial datum $F_I$ and the comparison principle for the heat equation.

\section{Deterministic market size fluctuations}\label{s:deterministic}
\noindent We introduce deterministic time dependent changes (fluctuations) of the size of the market in the Lasry and Lions model. 
This setting is the first step towards the development of a proper stochastic formulation, where mass fluctuations are  modeled by typical trajectories of semi-martingales. \\
Let $b_l = b_l(t)$ and $b_r = b_r(t)$ denote fluctuations for the buyer and vendor densities and assume that\\
\begin{enumerate}[label=(A\arabic*), start=2]
\item \label{a:alar} $b_l$, $b_r$ be in $C^{0,1}[0,\infty)$ with $b_l(0) = b_r(0) = 0$.\\
\end{enumerate}
\noindent Note that $b_l(0) = b_r(0) = 0$ can always be achieved by a constant shift $b_l$, $b_r$, which will not change the PDE model stated below.\\

\noindent Market size fluctuations are introduced conceptually by the following time splitting scheme, which corresponds to a numerical solution algorithm of a PDE: \vspace*{0.25em}\\
\begin{enumerate}[leftmargin=*]
\item Given a signed buyer-vendor distribution $f = f(x,t)$ at time $t$, solve the Lasry and Lions model \eqref{e:lasrylions} from time $t$ to $t+\Delta t$.
\item At the end of each time step modify the market by multiplying the distribution of buyers $f^+$ and vendors $f^-$ by $e^{b_l(t+\Delta t)-b_l(t)}$ and $e^{b_r(t+\Delta t)-b_r(t)}$ respectively. 
\end{enumerate}
\vspace*{0.5em}
This splitting scheme yields in the (formal) limit $\Delta t \rightarrow 0$, the problem
\begin{subequations}\label{e:ll2}
\begin{align}
f_t = k f_{xx} + \lambda(t)&(\delta(x-p(t)+a)-\delta(x-p(t)-a))\label{e:ll2equ}\\
+& (\dot{b}_l f^+ - \dot{b}_r f^-) \qquad \qquad \qquad \qquad ~ \text{ in } \mathbb{R} \times \mathbb{R}_+.\nonumber\\
\lambda(t) = - k f_x(p(t),t)&, ~f(p(t),t) = 0 \qquad \qquad \qquad \qquad \text{ in } \mathbb{R}_+\label{e:ll2lambda}\\
f(\cdot,0) = f_I \text{ and } &p(0) = p_0.
\end{align} 
\end{subequations}
Assuming that the free boundary $p = p(t)$ is uniquely defined by \eqref{e:ll2lambda}, we can compute the actual fluctuations of the numbers of buyers and vendors by integrating \eqref{e:ll2equ} over $(-\infty,p(t))$ and $(p(t),\infty)$, respectively and taking into account that $f > 0$ for $x < p(t)$ and $f < 0$ for $x > p(t)$. \\
Let $M^l(t) := \int_{\R} f^+(x,t) dx $ and $M^r(t) := \int_{\R} f^-(x,t)dx$. Using the definition of $\lambda$ in \eqref{e:ll2lambda}, we obtain
\begin{align*}
\frac{d}{dt} M^l(t) = \dot{b}_l(t) M^l(t),~~ \frac{d}{dt} M^r(t) = \dot{b}_r(t) M^r(t),
\end{align*}
and, therefore, for $ t \geq 0$,
\begin{align*}
 M^l(t) = e^{b_l(t)} M^l~~ \text{ and } ~~ M_r(t) = e^{b_r(t)} M^r .
\end{align*}
 Hence the proposed time splitting results in exponential market size changes for the buyers and vendors.

\noindent The rigorous justification of the convergence of the splitting scheme to \eqref{e:ll2} follows easily from the reformulation of the problem as the heat equation in each time interval $(t,t+\Delta t)$ after multiplying the initial datum at $t$ from the left and right by the appropriate exponential. We leave the details to the reader.\\
The results here are based on the transformation
connecting the Lasry and Lions model \eqref{e:lasrylions} in a one-to-one way to the heat equation \eqref{e:Fheat}. Since market size fluctuations are homogeneous in space the transformation
yields 
\begin{subequations} \label{e:F}
\begin{align}
  &F_t = k F_{xx} + \dot{b}_l F^+ -\dot{b}_r F^- ~\quad \text{ in } \mathbb{R} \times \mathbb{R}_+ \\
  &F(\cdot,0) = F_I \qquad \qquad \qquad \qquad \text{ in } \mathbb{R},
\end{align}
\end{subequations}
where the initial datum is given by \eqref{e:transfi}; note that $p = p(t)$ is the zero level set of the function $F$. We can rewrite equation \eqref{e:F}, using  $F^+ = \frac{F + \lvert F \rvert}{2}$ and $F^- = \frac{\lvert F \rvert - F}{2}$, as
\begin{align}\label{e:lldet}
F_t &= kF_{xx} + \dot{b}_l F^+ - \dot{b}_r F^- \\
&= k F_{xx} + \frac{1}{2}(\dot{b}_l + \dot{b}_r) F + \frac{1}{2}(\dot{b}_l- \dot{b}_r) \lvert F  \rvert. \nonumber
\end{align}
Then the exponential transformation $U =  F e^{-(b_l+b_r)/2}$ gives
\begin{subequations}\label{e:U}
\begin{align}
&U_t = k U_{xx} + \frac{1}{2}\left(\dot{b}_l - \dot{b}_r\right) \lvert U \rvert \quad \text{ in } \mathbb{R} \times \mathbb{R}_+ \label{e:U1}\\
&U(\cdot,0) = F_I \qquad \qquad \qquad \qquad \text{ in } \mathbb{R}.
\end{align}
\end{subequations}
We now state the following existence and uniqueness result for \eqref{e:F}.
\begin{thm}\label{t:exist}
Assume \eqref{a:finit} and \ref{a:alar}. Then the initial value problem \eqref{e:F} has a unique global-in-time classical solution and the
free boundary $p = p(t)$  is the graph of a locally bounded continuous function of time.
\end{thm}
\begin{proof}
From the Lipschitz continuity of the absolute value function we can deduce the existence of a unique solution of \eqref{e:U}. The free boundary $p$ corresponds to the zero level set of the function $U$, that is ~$U(p(t),t) = 0$. It follows from Angenent \cite{Angenent88} that the function $U(\cdot,t)$ has at most one zero $p(t)$ for every $t>0$, see also \cite{G2004}. In the former reference Angenent  
showed that the number of zeros of any solution $u$ of
\begin{align}\label{e:u1}
u_t = k u_{xx} + q~ u, \quad \text{ in } \in\mathbb{R} \times \mathbb{R}_+,
\end{align} with $q = q(x,t) \in L^{\infty}$ cannot increase in time. Since $\frac{\lvert U\rvert}{U}$ is bounded, Lemma 5.1 in \cite{Angenent88} guarantees that solutions to \eqref{e:U} cannot vanish identically on any interval $(x_0, x_1)$. It is, therefore, immediate that the free boundary has no ``fat'' parts.\\
\noindent Next we show that the free boundary cannot become unbounded in finite time.
Let $b(t) := \frac{1}{2} (\dot{b}_l(t) - \dot{b}_r(t))$, write \eqref{e:U1} in the form \eqref{e:u1} with $q(x,t) = b(t) \sgn U(x,t)$. 
We write $U= U_1-U_2$ where $U_1$ and $U_2$ are respectively the solutions of
\begin{align*}
\begin{aligned}
&\frac{\partial}{\partial t} U_1 = k \frac{\partial^2}{\partial x^2} U_1 + q U_1 \qquad \text{ in } \mathbb{R} \times \mathbb{R}_+, \\
&U_1(\cdot,0) = F_I^+ \qquad \qquad \qquad \text{ in } \mathbb{R}, 
\end{aligned}
\end{align*}
and 
\begin{align*}
\begin{aligned}
&\frac{\partial}{\partial t} U_2 = k \frac{\partial^2}{\partial x^2} U_2 + q U_2\qquad \text{ in } \mathbb{R} \times \mathbb{R}_+,\\
 &U_2(\cdot,0) = F_I^-\qquad \qquad \qquad \text{ in } \mathbb{R}.
\end{aligned}
\end{align*}
 For any given $T>0$ set $A_T:=\inf_{x \in \mathbb{R},~ t \in (0,T)} q(x,t)$, $B_T := \sup_{x\in \mathbb{R},~t \in (0,T)} q(x,t)$.
Since $U_1, U_2 \geq 0$, the comparison principle gives
\begin{align*}
e^{A_T t} R(x,t) \leq U_1(x,t) \leq e^{B_T t} R(x,t)~~\text{ and }~~e^{A_t t} S(x,t) \leq U_2(x,t) \leq e^{B_T t} S(x,t),
\end{align*}
where $R$ and $S$ are solution of the initial value problems
\begin{align*}
\begin{aligned}
&R_t = k R_{xx} \quad \text{ in } \mathbb{R} \times \mathbb{R}_+ \\
&R(\cdot,0) = F_I^+ ~\text{ in } \mathbb{R} 
\end{aligned}
\quad \text{ and } \quad
\begin{aligned}
&S_t = k S_{xx}\quad \text{ in } \mathbb{R} \times \mathbb{R}_ü\\
&S(\cdot,0) = F_I^- \text{ in } \mathbb{R}.
\end{aligned}
\end{align*}
Expressing $R$ and $S$ in terms of the fundamental solution of the heat equation yields that, for any $\delta > 0$, all $t \in (\delta, T]$, we have
\begin{align*}
&R(x,t) \rightarrow 0 \text{ as } x \rightarrow \infty \text{ and } S(x,t) \rightarrow 0 \text{ as } x \rightarrow -\infty,
\end{align*}
while $R(\cdot,t)$ and $S(\cdot,t)$ remain bounded away from $0$ as $x \rightarrow -\infty$ and $x \rightarrow \infty$ respectively see \cite{CMP2011} for full details. We conclude that the signs of $U$ in the far fields are determined by the signs of the initial datum $U_I$, which is positive on the left and negative on the right. The continuity of $U$ implies the existence of a zero in between and the free boundary $p$ stays locally bounded in time.
\end{proof}

\noindent It is also of interest to consider market size fluctuation functions with jump discontinuities $b_l$ and $b_r$, particularly with Levy-type random processes in hindsight. Assume for the
moment that, instead of \ref{a:alar}, we have:\\

\begin{enumerate}[label=(A\arabic*'), start=2]
\item There is an increasing sequence $\lbrace T_k \rbrace_{k \in \mathbb{N}},~ T_0 = 0$, $T_0^- := T_0^+ := 0$, $T_k \rightarrow \infty$ as $k \rightarrow \infty$ such that, for all $k \in \mathbb{N}$, $b_l,~b_r \in W^{1,\infty}(T_k, T_{k+1})$ with $b_l(0) = b_r(0) = 0$.\label{a:piececonst}\\
\end{enumerate}

\noindent We are interested in the case where either $b_r (T_k^+)$  or $b_l(T_k^+)$  is different from $b_r (T_k^-)$ or $b_l(T_k^-)$ respectively.
Then the limit $\Delta t \rightarrow 0$ of the splitting scheme has to be modified even on the formal level as follows:\\
On every interval $(T_k,T_{k+1})$ we solve \eqref{e:ll2} subject to the initial datum
\begin{align*}
f_I(x,T_k^+) = f_I(x,T_k^-) 
\begin{cases}
\exp(b_l(T_k^+)-b_l(T_k^-)),\quad \text{ if } x < p(T_k)\\
\exp(b_r(T_k^+)-b_r(T_k^-)),\quad \text{ if } x > p(T_k).
\end{cases}
\end{align*}
Note that Theorem \ref{t:exist} applies on each interval $(T_k,T_{k+1})$ and that the free boundary $p = p(t)$ is globally continuous.

\noindent Standard results for the (obvious) time splitting scheme of semilinear parabolic equations with smooth nonlinearities (see \cite{F2009}) yield the convergence of the scheme. 
It is immediate that the convergence is retained for less smooth nonlinearities (albeit possibly loosing the convergence order). Applying first the 
exponential transformation, which gives the splitting scheme for \eqref{e:F} and equivalently for \eqref{e:ll2}, and, consecutively, the transformation \eqref{e:invF}, we conclude the convergence of the solution of the time spitting scheme.

\section{The long time asymptotics}\label{s:asymptotics}

\noindent We now consider the asymptotic behavior of \eqref{e:ll2} for $t \rightarrow \infty$. We use the super- and sub-scripts $l$ and $r$ to refer to the left and right sides of the free boundary $p = p(t)$ in the following.\\
In what follows we assume that \\
\begin{enumerate}[label=(A\arabic*),ref=A\arabic*,start=3,]
\item\label{a:fluctuations} 
\begin{enumerate}[label=(\roman*),ref=(\theenumi)(\roman*)]
\item Let $b(t) := \frac{1}{2} (b_l(t) - b_r(t)) \in W^{1,\infty}(0,\infty) \cap W^{1,1}(0,\infty)$, \label{a:flucspaces}
\item  $\lim_{t\rightarrow \infty} b_l(t) = b_l^\infty$ and $ \lim_{t\rightarrow \infty} b_r(t) = b_r^\infty,$ \label{a:flucbound}
\item $\lvert b_l(t)-b_l^{\infty} \rvert + \lvert b_r(t)-b_r^{\infty} \rvert = \mathcal{O}(t^{-w})$ for all $t>0$  with $w > 1$, \\
or alternatively we shall use \label{a:flucdec1}
\item$\lvert b_l(t)-b_l^{\infty} \rvert + \lvert b_r(t)-b_r^{\infty} \rvert = \mathcal{O}((1+t)^{-w})$ for all $t>0$  with $w > \frac{3}{2}$. \label{a:flucdec2}\\
\end{enumerate}
\end{enumerate}

\noindent To simplify the argument, we write
\begin{align*}
e^{b_l(t)} := l(t) \text{ and } ~e^{b_r(t)} := r(t).
\end{align*}
Next we consider the exponential transformation of the original problem \eqref{e:lldet} given by 
\begin{align} \label{e:exptrans}
G(x,t) := 
\begin{cases}
e^{-b_l(t)} F(x,t)~\text{ if } x < p(t)\\
e^{-b_r(t)} F(x,t)~\text{ if }  x > p(t).
\end{cases}
\end{align}
It is immediate that $G$ satisfies the following problem
\begin{subequations}\label{e:G}
\begin{align}
  G_t &= k G_{xx}\qquad \qquad ~ \text{ in } \lbrace (x,t) \mid t > 0 \text{ and } x \neq p(t) \rbrace\\
  G(p(t),t) &= 0\phantom{G_{xx} } \qquad \qquad  ~ \text{ in }  \mathbb{R}_+,\\
l(t) G_x \mid_{x = p(t)^-} &= r(t) G_x \mid_{x = p^(t)^+} \text{ in } \mathbb{R}_+\\
G(\cdot,0) &= G_I = F_I.
\end{align}
\end{subequations}
To determine the asymptotic behavior in case of deterministic market fluctuations we use the parabolic rescaling and define
$x = \frac{y}{\varepsilon} ~\text{ and } t = \frac{\tau}{\varepsilon^2}.$
Then $G(y,\tau) := G(\frac{y}{\varepsilon}, \frac{\tau}{\varepsilon^2})$ satisfies
\begin{subequations}\label{e:H}
\begin{align}
  G_{\tau} &=k G_{yy} \qquad~ \qquad  \text{ in } \lbrace (x,\tau) \mid \tau > 0 \text{ and } x \neq q^{\varepsilon}(\tau) \rbrace\\
G^{\varepsilon}(q^{\varepsilon}(\tau),\tau) &= 0 \phantom{G_{yy}(y,\tau),} \qquad \text{ in } \mathbb{R}_+,\\
l(\frac{\tau}{\varepsilon^2}) G_y\mid_{y = q^{\varepsilon}(\tau)^-} &= r(\frac{\tau}{\varepsilon^2})G_y \mid_{y = q^{\varepsilon}(t)^+},\\
G(\cdot,0) &= F_I^{\varepsilon}.
\end{align}
\end{subequations}
We study the limit as $\varepsilon \rightarrow 0$ and note that $q^{\varepsilon(\tau)} = \varepsilon p(\frac{\tau}{\varepsilon^2})$. At first we assume that $q^{\varepsilon} = q^{\varepsilon}(\tau)$ is known (and smooth) and solve
\eqref{e:H} on the left and the right of the free boundary respectively. Let $G^{l}$ and $G^{r}$ be the smooth solutions of
\begin{align*}
\begin{aligned}
&  G_{\tau}^{l} =k G_{yy}^{l} \quad \text{ in } (-\infty,q^{\varepsilon}(\tau)),~\tau > 0\\
&G^{l}(q^{\varepsilon}(\tau),\tau) = 0  \quad \text{ in } \mathbb{R}_+,\\
&G^{l}(\cdot,0) = F_I^{\varepsilon} \quad \text{ in } (-\infty, q^{\varepsilon}(0))
\end{aligned}
\quad \text{ and } \quad
\begin{aligned}
&  G_{\tau}^{r} =k G_{yy}^{r} \quad~  \text{ in } (q^{\varepsilon}(\tau),\infty),~\tau > 0\\
&G^{r}(q^{\varepsilon}(\tau),\tau) = 0  \quad \text{ in } \mathbb{R}_+,\\
&G^{r}(\cdot,0) = F_I^{\varepsilon}\quad \text{ in } (q^{\varepsilon}(0), \infty).
\end{aligned}
\end{align*}
Then $G^{l}$ and $G^{r}$ can be written as
\begin{align*}
G^{l} (y,\tau) &= \int_{-\infty}^{q^{\varepsilon}(0)}  F_I^{\varepsilon} (\xi) K(y-\xi,\tau) d\xi  \\
&+ \int_0^{\tau} \e^{-b_l(\frac{s}{\varepsilon^2})}  F^{\varepsilon,l}_y(q^{\varepsilon}(s),s) K(y-q^{\varepsilon}(s), \tau-s) ds,\\
G^{r} (y,\tau) &= \int_{-\infty}^{q^{\varepsilon}(0)} F_I^{\varepsilon} (\xi) K(y-\xi,\tau) d\xi \\
&- \int_0^{\tau} \e^{-b_r(\frac{s}{\varepsilon^2})} F^{\varepsilon,l}_y(q^{\varepsilon}(s),s) K(y-q^{\varepsilon}(s), \tau-s) ds,
\end{align*}
where $F^{\varepsilon}(y,\tau) := F(\frac{y}{\varepsilon},\frac{\tau}{\varepsilon^2})$ and 
\begin{align}\label{e:heatkernel}
K(x,t) := \frac{1}{\sqrt{4 \pi k t}} e^{-\frac{x^2}{4 kt}}.
\end{align}
is the heat kernel. \\
Since $G^{l}(q^{\varepsilon}(\tau),\tau) = e^{-b_l(\frac{\tau}{\varepsilon^2})}F^{\varepsilon}(q^{\varepsilon}(\tau),\tau) = 0$  and $G^{r}(q^{\varepsilon}(\tau),\tau)= e^{-b_r(\frac{\tau}{\varepsilon^2})}F^{\varepsilon}(q^{\varepsilon}(\tau),\tau) = 0$ we find
\begin{align*}
 &0 = F^{\varepsilon}(q^{\varepsilon}(\tau),\tau) = \int_{-\infty}^{q^{\varepsilon}(0)}e^{b_l(\frac{\tau}{\varepsilon^2})} F_I^{\varepsilon} (\xi) K(q^{\varepsilon}(\tau)-\xi,\tau) d\xi \\
& \phantom{0 = F^{\varepsilon}(q^{\varepsilon}(\tau),\tau)} + \int_0^{\tau} e^{b_l(\frac{\tau}{\varepsilon^2}) - b_l(\frac{s}{\varepsilon^2})} F^{\varepsilon}_y(q^{\varepsilon}(s),s) K(q^{\varepsilon}(\tau)-q^{\varepsilon}(s), \tau-s) ds
\end{align*}
and
\begin{align*}
& 0 = F^{\varepsilon}(q^{\varepsilon}(\tau),\tau) =  \int_{-\infty}^{q^{\varepsilon}(0)} e^{b_r(\frac{\tau}{\varepsilon^2})} F_I^{\varepsilon} (\xi) K(q^{\varepsilon}(\tau)-\xi,\tau) d\xi\\
& \phantom{ 0 = F^{\varepsilon}(q^{\varepsilon}(\tau),\tau)} - \int_0^{\tau} e^{b_r(\frac{\tau}{\varepsilon^2}) - b_r(\frac{s}{\varepsilon^2})} F^{\varepsilon}_y(q^{\varepsilon}(s),s) K(q^{\varepsilon}(\tau)-q^{\varepsilon}(s), \tau-s) ds  . 
\end{align*}
Adding the above equations gives $\mathscr{I}_{1,\varepsilon} + \mathscr{I}_{2,\varepsilon} = 0$ with
\begin{subequations}\label{e:I}
\begin{align}
\begin{split}\label{e:I1eps}
&\mathscr{I}_{1,\varepsilon} := \\
&~~~ \int_0^{\tau} \bigl[e^{b_l(\frac{\tau}{\varepsilon^2})-b_l(\frac{s}{\varepsilon^2})}-e^{b_r(\frac{\tau}{\varepsilon^2})-b_r(\frac{s}{\varepsilon^2})} \bigr]F_y^{\varepsilon}(q^{\varepsilon}(s),s) K(q^{\varepsilon}(\tau)-q^{\varepsilon}(s),\tau-s) ds \\
\end{split}
\end{align}
and
\begin{align}
\begin{split}\label{e:I2eps}
 &\mathscr{I}_{2,\varepsilon} := \\
& ~\int_{-\infty}^{q^{\varepsilon}(0)} e^{b_l(\frac{\tau}{\varepsilon^2})} F_I^{\varepsilon}(\xi) K(q^{\varepsilon}(\tau)-\xi,\tau) d\xi + \int^{\infty}_{q^{\varepsilon}(0)} e^{b_r(\frac{\tau}{\varepsilon^2})} F_I^{\varepsilon}(\xi) K(q^{\varepsilon}(\tau)-\xi,\tau) d\xi \\
\end{split}
\end{align}
\end{subequations}

\noindent We shall use the following estimate in order to compare $\mathscr{I}_{1,\varepsilon}$ and $\mathscr{I}_{2,\varepsilon}$ asymptotically for $\varepsilon \rightarrow 0$.
\begin{lem}
Assume \ref{a:flucspaces} and $F_I \in W^{1,\infty}(\mathbb{R})$. Then
\begin{align*}
\lVert \frac{d}{dx}U(\cdot,t) \rVert_{L^{\infty}(\mathbb{R})} \leq e^{\int_0^{\infty} \lvert \dot{b}(s) \rvert ds} \lVert \frac{d}{dx} F_I\rVert_{L^{\infty}(\mathbb{R})},
\end{align*}
where $U$ solves \eqref{e:U}.\\
\end{lem}

\noindent Note that this implies, that for some $C > 0$, $\lVert F_x \rVert_{L^{\infty}(\mathbb{R} \times (0,\infty))} \leq C$ and thus gives the $L^{\infty}(0,\infty)$-bound of $\lambda = \lambda(t)$ used in the proof of Theorem \ref{t:thm1} and \ref{t:thm2}.\\
\begin{proof}
We write $W = U_x$ and obtain by differentiating  \eqref{e:U}
\begin{subequations}\label{e:W}
\begin{align}
  W_t &= k W_{xx} + \dot{b}(t) \sgn(U) W \text{ in } \mathbb{R} \times \mathbb{R},\\
W(\cdot,0) &= \frac{d}{dx}F_I. 
\end{align}
\end{subequations}
Multiplying \eqref{e:W} by $W^{2j-1}$ for some fixed $j > 1$ and integrating over $\mathbb{R}$ yields
\begin{align*}
\frac{1}{2j}\frac{d}{dt} \int_{\R} W^{2j} dx = -k(2j-1)\int_{\R}W_x^2W^{2j-2} dx  + \dot{b}(t) \int_{\R}\sgn(U)W^{2j} dx,
\end{align*}
and, therefore,
\begin{align*}
\frac{d}{dt} \int_{\R} W^{2j}dx \leq 2j \lvert \dot{b}(t) \rvert \int_{\R}W^{2j}dx,
\end{align*}
and, hence,
\begin{align*}
  \lVert W(\cdot,t) \rVert_{L^{2j}(\R)} \leq e^{\int_0^t \lvert \dot{b}(s) \rvert ds }\lVert \frac{d}{dx} F_I \rVert_{L^{2j}(\R)}.
\end{align*}
This results letting the limit $j\rightarrow \infty$.
\end{proof}

\vspace*{1em}

\noindent We remark that $F_I \in W^{1,\infty}(\R)$ holds if, for example $f_I \in W^{1,\infty}(\R)$ has compact support. \\

\noindent Next we consider the rescaled initial datum in the limit $\varepsilon \rightarrow 0$.
\begin{prop}\label{p:Finit}
  Let $F_I = F_I(x)$ be given by \eqref{e:transfi}. Then the rescaled initial datum satisfies $F_I(\frac{x}{\varepsilon}) = F_I^0(x) + \frac{\varepsilon}{2} (M^l - M^r) \delta(x)
+ \frac{\varepsilon}{a} (\int_{-\infty}^0 z f_I^+(z) dz + \int_0^{\infty} z f_I^-(z) dz ) \delta (x) + \mathcal{O}(\varepsilon^2)$ in $\mathcal{D}'(\R)$, where 
\begin{align*}
F_I^0(x)&= \begin{cases}
\frac{1}{a} M^l &\text{ for } x < q(0) = 0\\
-\frac{1}{a} M^r &\text{ for } x > q(0) = 0.
\end{cases}
 \end{align*}
\end{prop}
\begin{proof}
It is immediate from the definitions of $F_I$ that, for every test function $\varphi \in C_0^\infty(\mathbb{R})$: 
\begin{align}\label{e:asyminit}
\begin{split}
  \int_{-\infty}^{\infty}&\left[ \sum_{n=0}^{\infty} f_I^+(\frac{y}{\varepsilon}+na) - \sum_{n=0}^{\infty} f_I^-(\frac{y}{\varepsilon} - na)\right] \varphi(y) dy =\\
&\varepsilon \int_{-\infty}^0 f_I^+(z)\sum_{n=0}^{\infty}  \varphi(\varepsilon(z-na)) dz - \varepsilon \int_0^{\infty} f_I^-(z)\sum_{n=0}^{\infty}  \varphi(\varepsilon(z+na)) dz.
\end{split}
\end{align}
It follows that
\begin{align}
\begin{split}\label{e:asyminit2}
\varepsilon \sum_{n=0}^{\infty}  \varphi(\varepsilon(z-na)) &= \varepsilon \sum_{n=0}^{\infty} (\varphi(-\varepsilon na) + \varepsilon z \varphi_x(-\varepsilon na)) + \mathcal{O}(\varepsilon^2)\\
&= \frac{1}{a} \int_{-\infty}^{\frac{\varepsilon a}{2}} \varphi(x) dx + \frac{\varepsilon z}{a} \int_{-\infty}^{\frac{\varepsilon a}{2}} \varphi_x(x) dx +  \mathcal{O}(\varepsilon^2)\\
&= \frac{1}{a} \int_{-\infty}^0 \varphi(x) dx + \frac{\varepsilon}{2} \varphi(0) + \frac{\varepsilon z}{a} \varphi(0) + \mathcal{O}(\varepsilon^2).
\end{split}
\end{align}
 Note that we split both integrals in the second line of \eqref{e:asyminit2} into two parts. The first one is approximated by the rectangle integration rule, while the contributions from the second interval over $(0,\frac{\varepsilon a}{2})$ were included in the $\mathcal{O}(\varepsilon^2)$ terms. \\
Similar calculations give $\varepsilon \sum_{n=0}^{\infty}  \varphi(\varepsilon(z+na)) = \frac{1}{a} \int^{\infty}_0 \varphi(x) dx + \frac{\varepsilon}{2} \varphi(0) - \frac{\varepsilon z}{a} \varphi(0) + \mathcal{O}(\varepsilon^2)$, which concludes the proof.
\end{proof}

\noindent Next we introduce the additional notation
\begin{align}
l^{\infty} = e^{b_l^{\infty}}, ~r^{\infty} = e^{b_r^{\infty}}, ~\alpha^l = \frac{l^{\infty}}{a} M^l \text{ and } \alpha^r = \frac{r^{\infty}}{a} M^r.
\end{align} 


\begin{thm}\label{t:thm1}
Assume \ref{a:flucspaces}, \ref{a:flucbound} and \ref{a:flucdec1}. Let $f_I$ be such that $F_I \in W^{1,\infty}(\mathbb{R})$. 
Then, as $t\rightarrow \infty$,
\begin{align*}
p(t)  = (\beta_1  + \mathcal{O}(\frac{1}{\sqrt{t}}) \sqrt{t},
\end{align*}
where $\beta_1 = \sqrt{4 k}\erf^{-1}(\frac{\alpha^l-\alpha^r}{\alpha^l+\alpha^r})$.
\end{thm}
\begin{proof} Note that since $kF_y^{\varepsilon}(q^{\varepsilon}(\tau),\tau) = \frac{1}{\varepsilon} \lambda(t)$, the bound on $\lambda$ that was obtained earlier yields, 
\begin{align*}
\mathscr{I}_{1,\varepsilon} \lesssim \frac{c}{\varepsilon} \lVert \lambda \rVert_{L^{\infty}(0,\infty)}  \mathscr{J}_{\varepsilon},
\end{align*}
with
\begin{align*}
  \mathscr{J}_{\varepsilon} & := \int_0^{\tau} \frac{1}{\sqrt{\tau-s}}~ \lvert e^{b_l(\frac{\tau}{\varepsilon^2})-b_l(\frac{s}{\varepsilon^2})}-e^{b_r(\frac{\tau}{\varepsilon^2})-b_r(\frac{s}{\varepsilon^2})}\rvert~ ds.
\end{align*} 
Since
\begin{align*}
\mathscr{J}_{\varepsilon}&= \varepsilon \int_0^t \frac{1}{\sqrt{t-u}}~  \lvert e^{b_l(t)-b_l(u)}-e^{b_r(t)-b_r(u)}\rvert~ du\\
&:=\varepsilon \mathscr{S}_t \quad \text{ with } t = \frac{\tau}{\varepsilon^2}
\end{align*}
we find $\mathscr{I}_{1,\varepsilon} \lesssim \lVert \lambda \rVert_{L^{\infty}(0,\infty)} \mathscr{S}_t$.\\
Consider now the $\frac{1}{2}$-fractional integral
\begin{align*}
(I_{\frac{1}{2}} f)(t) &= 
\int_0^{t} \frac{f(s)}{\sqrt{t-s}}~ ds.
\end{align*}
It is easy to verify that, if $f(t)\lesssim t^{-w_1}$ for $w_1 > 0$
\begin{align*}
(I_{\frac{1}{2}} f)(t) \lesssim t^{-\frac{1}{2}} + \lvert \ln t \rvert t^{-w_1+\frac{1}{2}}.
\end{align*}
Then the asymptotic assumption \ref{a:flucdec1} on $b_l$ and $b_r$ implies
\begin{align}\label{e:I1asym}
\mathscr{J}_{1,\varepsilon} = \mathcal{O}(\varepsilon) \text{ as } \varepsilon \rightarrow 0.
\end{align}
and, therefore
\begin{align*}
\mathscr{I}_{2,\varepsilon} = \int_{-\infty}^{q^{\varepsilon}(0)} e^{b_l^\infty} F_I^{\varepsilon}(\xi) K(q^{\varepsilon}(\tau)-\xi,\tau) d\xi + \int^{\infty}_{q^{\varepsilon}(0)} e^{b_r^\infty} F_I^{\varepsilon}(\xi) K(q^{\varepsilon}(\tau)-\xi,\tau) d\xi + \mathcal{O}(\varepsilon^2).
\end{align*}
Since, in view of Proposition \eqref{p:Finit},
\begin{align*}
\begin{aligned}
\begin{rlcases}
e^{ b_l^{\infty}} F_I^{\varepsilon}(x),~ x < 0\\
e^{ b_r^{\infty}} F_I^{\varepsilon}(x),~ x > 0
\end{rlcases}
\end{aligned}
=
\begin{aligned}
\begin{rlcases}
  \frac{1}{a}e^{b_l^{\infty}} M^l,~ x < 0\\
 \frac{1}{a}e^{b_l^{\infty}} M^r,~x < 0
\end{rlcases}
\end{aligned}
+\mathcal{O}(\varepsilon) \text{ in } \mathcal{D}'(\mathbb{R}),
\end{align*}
using that $q^{\varepsilon}(0) = 0$, we have
\begin{align*}
\mathscr{I}_{2,\varepsilon} = \alpha^l \int_{-\infty}^0 K (q^{\varepsilon}(\tau)-\xi, \tau) d\xi - \alpha^r \int_0^{\infty} K(q^{\varepsilon}(\tau)-\xi,\tau) d\xi + \mathcal{O}(\varepsilon).
\end{align*}
This yields together with \eqref{e:I1asym} the equation for $q^{\varepsilon}(\tau)$:
\begin{align*}
\alpha^l \int_{-\infty}^0 K(q^{\varepsilon}(\tau)-\xi, \tau) d\xi - \alpha^r \int_0^{\infty} K(q^{\varepsilon}(\tau)-\xi,\tau) d\xi = \mathcal{O}(\varepsilon).
\end{align*}
Hence we obtain
\begin{align*}
-\frac{\alpha^l}{2} (1+\erf(\frac{q^{\varepsilon}(\tau)}{\sqrt{4 k \tau}})) + \frac{\alpha^r}{2} (1-\erf(\frac{q^{\varepsilon}(\tau)}{\sqrt{4 k \tau}})) = \mathcal{O}(\varepsilon),
 \end{align*}
and thus
\begin{align*}
\erf(\frac{q^{\varepsilon}(\tau)}{\sqrt{4 k \tau}}) = \frac{\alpha^l - \alpha^r}{\alpha^l+\alpha^r} + \mathcal{O}(\varepsilon).
\end{align*}
The claim now follows 
\begin{align*}
q^{\varepsilon}(\tau) = \beta_1 \sqrt{\tau} + \mathcal{O}(\varepsilon)\sqrt{\tau},
\end{align*}
which, in view of the rescaling, yields
\begin{align*}
p(t) = \frac{q^{\varepsilon}(\varepsilon^2 t)}{\varepsilon} = \beta_1 \sqrt{t} + \mathcal{O}(1) \text{ as } t \rightarrow \infty.
\end{align*}
\end{proof}
\vspace*{1em}

\noindent For the case $\alpha^l = \alpha^r$ we shall make use of the following auxiliary Lemma:
\begin{lem}\label{l:aux}
Assume $f \in W^{1,\infty}(0,\infty)$ and $\vert f(t) \rvert \leq \frac{c}{(1+t)^\sigma}$ for $t > 0$ with $\sigma > 0,~c>0$. If $\sigma> \frac{3}{2}$, there
exists a function $H \in L^1_+(0,\infty) \cap L^\infty_+(0,\infty)$ such that, for $s \in (0,t)$
\begin{align}
\lvert f(t)-f(s) \rvert \leq \frac{H(s)}{\sqrt{t}}\sqrt{t-s}.
\end{align}
\end{lem}

\begin{proof}
We rewrite 
\begin{align*}
\lvert f(t)-f(s) \rvert &= \sqrt{\lvert f(t)-f(s) \rvert}\sqrt{\lvert f(t)-f(s) \rvert}\\
&\leq \sqrt{Lip(f)}{\sqrt{t-s}} \left(\frac{c}{(1+t)^\sigma} + \frac{c}{(1+s)^\sigma}\right).
\end{align*}
Let $\alpha \in (0,1)$ and consider $s\in [\alpha t, t]$. Then
\begin{align*}
\lvert f(t)-f(s)\rvert \leq c \sqrt{Lip(f)} \sqrt{t-s}\frac{2}{(1+\alpha t)^{\sigma}}.
\end{align*}
Since $\min(a,b) \leq \sqrt{ab}$ for all $a,~b>0$, we have for $s\in[\alpha t, t]$
\begin{align*}
\lvert f(t)-f(s) \rvert \leq C \sqrt{t-s}\frac{1}{(1+\alpha t)^{\frac{\sigma}{2}}} \frac{1}{(1+s)^{\frac{\sigma}{2}}},
\end{align*}
and thus if $t >1$ and  $s \in [\alpha t, t]$
\begin{align*}
\lvert f(t)-f(s) \rvert \leq C_1 \sqrt{t-s}\frac{1}{\sqrt{t}} \frac{1}{(1+ \alpha s)^{\sigma - \frac{1}{2}}}.
\end{align*}
Since, for all $s \in [0, \alpha t]$,
\begin{align*}
\frac{\sqrt{t-s}}{\sqrt{t}} \in [\sqrt{1-\alpha},1],
\end{align*}
we have
\begin{align*}
\lvert f(t)-f(s)\rvert \leq \frac{\sqrt{t-s}}{\sqrt{t}} \frac{1}{\sqrt{1-\alpha}} \frac{2c}{(1+s)^{\sigma}}.
\end{align*}
The result follows with $H(s) = C_2 \frac{1}{(1+\alpha s)^{\sigma - \frac{1}{2}}}$.
\end{proof}

\noindent The next result is about the long time asymptotics of $p$ when $\alpha^r = \alpha^l$.
\begin{thm}\label{t:thm2}
Assume \ref{a:flucspaces}, \ref{a:flucbound} and \ref{a:flucdec2} and let $f_I$ be such that $F_I \in W^{1,\infty}(\mathbb{R})$. If $\alpha^l = \alpha^r$ 
\begin{align*}
p(t) = \beta_2 + \mathcal{o}(1) \text{ as } t \rightarrow \infty,
\end{align*}
with 
\begin{align}\label{e:beta2}
\beta_2 = \frac{1}{2 \alpha a}\left[l^{\infty}\int_{-\infty}^0 z f_I^+(z) dz + r^{\infty}\int_0^{\infty} z f_I^-(z) dz - a \int_0^{\infty} \zeta(\infty,s) \lambda(s) ds\right].
\end{align}
\end{thm}
\begin{proof}
We revisit the equation $\mathscr{I}_{1,\varepsilon} + \mathscr{I}_{2,\varepsilon} = 0$ and refine the asymptotics for each term.
Recalling \eqref{e:I1eps} we find:
\begin{align*}
\mathscr{I}_{1,\varepsilon} &= - \int_0^{t} \left( e^{b_l(t)-b_l(s)}-e^{b_r(t)-b_r(s)}\right) \lambda(s) \frac{1}{\sqrt{4 k\pi (t-s)}}e^{-\frac{(p(t)-p(s))^2}{4k(t-s)}} ds \\
&= - \int_0^\infty \mathbb{1}_{\lbrace s < t \rbrace} \underbrace{\left(e^{b_l(t)-b_l(s)}-e^{b_r(t)-b_r(s)}\right)}_{:= \zeta(t,s)}  \lambda(s)\frac{1}{\sqrt{4 k \pi (t-s)}}e^{-\frac{(p(t)-p(s))^2}{4k(t-s)}} ds \\
&= -\frac{1}{\sqrt{4 k \pi t}}  \int_0^{\infty} \mathbb{1}_{\lbrace s < t \rbrace } \zeta(t,s) \lambda(s)\frac{1}{\sqrt{1-\frac{s}{t}}} e^{-\frac{(p(t)-p(s))^2}{4k(t-s)}} ds,
\end{align*}
where we used that $kF^{\varepsilon}_y(q^{\varepsilon}(\tau),\tau) = -\frac{1}{\varepsilon}\lambda(t)$ and $t = \frac{\tau}{\varepsilon^2}$.
Since in view of Theorem \ref{t:thm1} $p(t) = \mathcal{O}(1)$ as $t\rightarrow \infty$  we have
\begin{align*}
\frac{(p(t)-p(s))^2}{4k(t-s)} \lesssim \frac{1}{(t-s)},
\end{align*}
and, thus, pointwise for a.e. $s > 0$, as $t \rightarrow \infty$ we get 
\begin{align*}
  \mathbb{1}_{\lbrace s < t \rbrace} \zeta(t,s) \lambda(s) \frac{1}{\sqrt{1-\frac{s}{t}}} e^{-\frac{(p(s)-p(t))^2}{4k(t-s)}} \rightarrow \zeta(\infty,s) \lambda(s).
\end{align*}
Since Lemma \ref{l:aux} ensures that $\lvert \zeta(t,s)\rvert \leq \sqrt{1-\frac{s}{t}} H(s)$ for $0< s < t$ we have
\begin{align*}
\mathbb{1}_{\lbrace s < t \rbrace} \lvert \zeta(t,s)\rvert \lambda(s) \frac{1}{\sqrt{1-\frac{s}{t}}} e^{-\frac{(p(t)-p(s))^2}{4k(t-s)}} \leq \lVert \lambda \rVert_{L^{\infty}} H(s),
\end{align*}
and, by the Lebesgue dominated convergence, we find that
\begin{align*}
\mathscr{I}_{1,\varepsilon} = -\left(\frac{1}{\sqrt{4k\pi t}}\int_0^{\infty} \zeta(\infty,s) \lambda(s) ds + \mathcal{o}(1)\right).
\end{align*}
To refine the asymptotics of $\mathscr{I}_{2,\varepsilon}$ given by \eqref{e:I2eps} we use Proposition \ref{p:Finit} again with $f_I$ replaced by $l^{\infty} f_I(x)$ for $x <0$ and $r^{\infty} f_I(x)$ for $x > 0$.
Letting $\alpha := \alpha^l = \alpha^r$, we get
\begin{align*}
 \mathscr{I}_{2,\varepsilon} &= -\frac{\alpha}{2}\Bigl(1+\erf\left(\frac{q^{\varepsilon}(\tau)}{\sqrt{4k\tau}}\right)\Bigr) + \frac{\alpha}{2} \Bigl(1-\erf \left(\frac{q^{\varepsilon}(\tau)}{\sqrt{4 k \tau}}\right)\Bigr)\\
&+ \frac{\varepsilon}{a}\Bigl[l^\infty \int_{-\infty}^0 z f^+_I(z) dz + r^{\infty} \int_0^{\infty} z f_I^-(z) dz\Bigr]K(q^{\varepsilon}(\tau),\tau) + \mathcal{o}(\varepsilon).
\end{align*}
Hence $\mathscr{I}_{1,\varepsilon} + \mathscr{I}_{2,\varepsilon} = 0$ becomes
\begin{align*}
-\alpha \erf\left(\frac{q^{\varepsilon}}{\sqrt{4 k \tau}}\right) + \frac{\varepsilon}{a\sqrt{4 k \pi \tau}} e^{\frac{-(q^{\varepsilon}(\tau))^2}{4 k \tau}} \left[l^\infty \int_{-\infty}^0 z f^+_I(z) dz + r^{\infty} \int_0^{\infty} z f_I^-(z) dz\right]&\\
-\frac{\varepsilon}{\sqrt{4 k \pi \tau}}\int_0^{\infty} \zeta(\infty,s) \lambda(s) ds  &= \mathcal{o}(\varepsilon).
\end{align*}
We conclude 
\begin{align*}
q^{\varepsilon}(\tau)  = (\beta_2 + \mathcal{o}(1))\varepsilon \text{ uniformly on compact intervals of } \tau \text{ as } \varepsilon \rightarrow 0,
\end{align*}
with $\beta_2 = \frac{1}{2 \alpha a}\left[l^{\infty}\int_{-\infty}^0 z f_I^+(z) dz + r^{\infty}\int_0^{\infty} z f_I^-(z) dz - a \int_0^{\infty} \zeta(\infty,s) \lambda(s) ds\right]$.\\
Since $p(t) = \frac{q^{\varepsilon}(\varepsilon^2 t)}{\varepsilon}$ we obtain the claim.
\end{proof}

\subsection{Periodic market size fluctuations}
\noindent We now discuss the price formation model in case of periodic fluctuations, that is we assume for all $s \in [0,\infty)$, there exist $\tau_l>0$ and $\tau_r>0$ such that
\begin{align*}
b_l(s) &= b_l(s + \tau_l)  \text{ and } b_r (s) = b_r(s + \tau_r).
\end{align*}
At first we study the highly oscillatory case on time intervals of length $\mathcal{O}(1)$, that is we choose $\varepsilon > 0$ small and replace  $b_l(t)$, $b_r(t)$ in 
\eqref{e:ll2} (and \eqref{e:F}, \eqref{e:lldet} and \eqref{e:U}) respectively by $b^{\varepsilon}_l(t) := b_l(\frac{t}{\varepsilon^2})$, and $b_r^{\varepsilon}(t):=b_r(\frac{t}{\varepsilon^2})$. The PDE is
then studied on the time interval $(0,T)$, where $0 < T < \infty$ is $\varepsilon$-independent.\\
 Then the IVP \eqref{e:U} reads: 
\begin{subequations}\label{e:Ueps}
\begin{align}
&U_t^{\varepsilon} = k U_{xx}^{\varepsilon} + \frac{1}{2\varepsilon^2}(\dot{b}_l(\frac{t}{\varepsilon^2}) - \dot{b}_r(\frac{t}{\varepsilon^2})) \lvert U^{\varepsilon} \rvert \text{ in } \mathbb{R} \times [0,T]\\
&U^{\varepsilon}(x,t=0) = F_I(x).
\end{align}
\end{subequations}
Note that the exponential transformation \eqref{e:exptrans} of equation \eqref{e:U} gives:
\begin{subequations}\label{e:Geps}
\begin{align}
G_t^{\varepsilon}(x,t) &= k G_{xx}^{\varepsilon}(x,t),~~ x \neq p^{\varepsilon}(t)\\
G^{\varepsilon}(p^{\varepsilon}(t),t) &= 0,\phantom{G_{xx}^{\varepsilon}(x,t) }~~ t > 0\\
l(\frac{t}{\varepsilon^2}) G_x^{\varepsilon} \mid_{x = p^{\varepsilon}(t)^-} &= r(\frac{t}{\varepsilon^2}) G_x^{\varepsilon} \mid_{x = p^{\varepsilon}(t)^+}\label{e:contGeps}\\
G^{\varepsilon}(x,t=0) &= F_I(x).
\end{align}
\end{subequations}
Assume that $\frac{\tau_l}{\tau_r}$ is rational, $r$ is uniformly positive and $\frac{l}{r}$ is not constant. Then, writing 
\eqref{e:contGeps} as
\begin{align*}
\frac{l(\frac{t}{\varepsilon^2})}{r(\frac{t}{\varepsilon^2})} = \frac{G_x^{\varepsilon}(p^{\varepsilon}(t)^+,t)}{G_x^{\varepsilon}(p^{\varepsilon}(t)^-,t)}.
\end{align*} 
We conclude that $G_x^{\varepsilon}(p^{\varepsilon}(t)^+,t) / G_x^{\varepsilon}(p^{\varepsilon}(t)^-,t)$ converges weakly to a constant (the average of
$\frac{l(\tau)}{r(\tau)}$ over the period) and thus $G_x^{\varepsilon}(p^{\varepsilon}(t)^+,t)$ and $G_x^{\varepsilon}(p^{\varepsilon}(t)^-,t)$ cannot both converge strongly. 
Therefore it is not possible to pass to the limit $\varepsilon \rightarrow 0$ in \eqref{e:contGeps} directly, that is the most 
basic formal argument fails.\\
Next we consider the problem  
\begin{align}
&U_t^{\varepsilon} = k U_{xx}^{\varepsilon} + \frac{1}{\varepsilon^2}\dot{b}(\frac{t}{\varepsilon^2}) \mathcal{f}(\Ueps) \quad \text{ in } \mathbb{R} \times \mathbb{R}_+\label{e:Ueps3}\\
&U^{\varepsilon}(x,t=0) = F_I(x), \nonumber
\end{align}
with $\mathcal{f} = \mathcal{f}(u)$ a smooth, non-negative function with $\mathcal{f}(0) = 0$ and $b$ a $1$-periodic function with $\langle  b \rangle = \int_0^1 b(\tau) d\tau = 0$; here we do
not assume $b(0) = 0$ but instead $b(\tau) = \frac{1}{2}\left(b_l(\tau)-b_r(\tau) - (\langle b_l \rangle - \langle b_r \rangle)\right)$.\\
We make the following asymptotic ansatz
\begin{align}\label{e:Ueps4}
\Ueps = \Phi(b(\frac{t}{\varepsilon^2}), V) + \varepsilon^2 W(x,t, \frac{t}{\varepsilon^2}) + \mathcal{O}(\varepsilon^3),
\end{align}
with $W$ being a $1$ periodic function in $\tau = \frac{t}{\varepsilon^2}$ and $\Phi = \Phi(s,\xi)$ solving 
\begin{subequations}\label{e:Phi}
\begin{align}
&  \Phi_s = \mathcal{f}(\Phi) \quad \text{ for } s\in \R\\
&\Phi(s=0, \xi) = \xi.
\end{align}
\end{subequations}
Noting that $\Phi_{\xi}$ satisfies
\begin{align*}
&(\Phi_{\xi})_s = \mathcal{f}'(\Phi) \Phi_{\xi}, \text{ and } ~\Phi_{\xi}(s,0,\xi) = 1 \, ,
\end{align*}
and we conclude that, for all $s > 0$ and $\xi \in \mathbb{R}$, 
\begin{subequations}
\begin{align}\label{e:Phixis}
\Phi_{\xi}(s,\xi) = \exp(\int_0^s \mathcal{f}'(\Phi(r,\xi))dr)
\end{align}
and
\begin{align}\label{e:Phixixis}
\Phi_{\xi\xi}(s,\xi) = \Phi_{\xi}(s, \xi) \int_0^s \mathcal{f}''(\Phi(r,\xi))\Phi_{\xi}(r,\xi) dr \, .
\end{align}
\end{subequations}
Next we  determine necessary conditions for the asymptotics \eqref{e:Ueps4} to hold. 
We start by computing
\begin{align*}
&\Ueps_t = \Phi_{\xi} V_t + \frac{1}{\varepsilon^2} \Phi_s\dot{b}  + \varepsilon^2 W_t + \dot{W} + \mathcal{O}(\varepsilon^3), ~\Ueps_x = \Phi_{\xi} V_x + \varepsilon^2 W_x + \mathcal{O}(\varepsilon^3) \\
&\Ueps_{xx} = \Phi_{\xi \xi} (V_x)^2 + \Phi_{\xi} V_{xx} + \varepsilon^2 W_{xx} + \mathcal{O}(\varepsilon^3). 
\end{align*}
Substituting $\Ueps$ and its derivatives into \eqref{e:Ueps} gives
\begin{align*}
\Phi_{\xi}(V_t-V_{xx}) - \Phi_{\xi \xi} (V_x)^2 = &- (\dot{W} + \dot{b} W\mathcal{f}'(\Phi)) - \varepsilon^2(W_t -W_{xx} -\dot{b}(W) \mathcal{f}''(\Phi + \sigma \varepsilon^2 W)),
\end{align*} 
where we have used the Taylor expansion of the term $\mathcal{f}(\Phi + \varepsilon^2 W)$ and ignored the $\mathcal{O}(\varepsilon^2)$ term. Note that
\begin{align*}
\dot{W} + \dot{b} W \mathcal{f}'(\Phi)  = \frac{1}{g} \frac{d}{d\tau} (g W),
\end{align*}
with $g(\tau, x, t) := \exp(-\int_0^{\tau} \dot{b}(\psi) \mathcal{f}'(\Phi(b(\psi), V(x,t)))  d\psi)$. Since
\begin{align*}
\int_0^{\tau} \dot{b}(\psi) \mathcal{f}'(\Phi(b(\psi), V(x,t)))  d\psi = \int_{b(0)}^{b(\tau)} \mathcal{f}'(\Phi(w, V(x,t))) dw
\end{align*}
we can rewrite $g$ as 
\begin{align*}
g(\tau, x, t) = \exp(-\int_ {b(0)}^{b(\tau)}  \mathcal{f}'(\Phi(w, V(x,t)))  dw)
\end{align*}
 and deduce that $\tau \rightarrow W g$  is $1$-periodic as well. Now we take the mean. This gives
\begin{align}\label{e:Phi2}
 \langle g \Phi_{\xi} \rangle (V_t - V_{xx}) - \langle g \Phi_{\xi \xi} \rangle (V_x)^2 = 0.
\end{align}
From \eqref{e:Phixis} we deduce that $\frac{1}{g(\tau,x,t)} = \frac{\Phixi(b(\tau),V)}{\Phixi(b(0), V)}$, while \eqref{e:Phixixis} gives 
\begin{align}
&\langle g \Phi_{\xi \xi} \rangle = \Phixi(b(0),V) \partial_{\xi} \int_0^1 \int_0^{b(\tau)} \mathcal{f}'(\Phi(r,V)) dr d\tau.\label{e:Phixixi}
\end{align}
Substituting \eqref{e:Phixixi} in \eqref{e:Phi2} and dividing by $\Phi_{\xi}(b(0),V)$ gives the necessary condition for periodicity of $\Weps(x,t,\tau)$ in the $\tau$-variable, that is
\begin{align*}
V_t = V_{xx} + \partial_x \langle \int_0^b \mathcal{f}'(\Phi(r,V)) dr \rangle V_x = 0.
\end{align*}

\noindent In the following we present a formal argument in the case $\mathcal{f}(U) = \frac{1}{2} \lvert U \rvert$. Since $\mathcal{f}'(U) = \frac{1}{2} \sgn(U)$ we have (as computed before)
\begin{align*}
\Phi(s,\xi) = e^{\frac{1}{2} s \sgn \xi} \xi \quad \text{ and } \quad \int_0^b \mathcal{f}'(\Phi(r,V)) dr = \frac{1}{2}b(\tau) \sgn V.
\end{align*}
We formally obtain that $V_t = V_{xx}$. Thus 
\begin{align*}
\Ueps(x,t) \sim e^{\frac{1}{2}(b(\frac{t}{\varepsilon^2}) ) \sgn V(x,t)} V(x,t),
\end{align*}
where $V$ solves the heat equation with initial datum
\begin{align}\label{e:VI}
V(x,t=0) &= 
\begin{cases}
e^{-\frac{1}{4}(b_l(0)-b_r(0) - \langle b_l \rangle + \langle b_r \rangle)} ~U_I(x) & x < p(0)\\
e^{~\frac{1}{4}(b_l(0)-b_r(0) - \langle b_l \rangle + \langle b_r \rangle)} ~U_I(x) & x > p(0).
\end{cases}
\end{align}

\noindent By using the nonlinear transformation \eqref{e:invF} from the heat equation to \eqref{e:lasrylions} it follows that the limiting free boundary of the Lasry-Lions model with periodic market fluctuations is given by the free boundary of the original Lasry-Lions model without market fluctuations but with a changed initial buyer-vendor distribution. $f_I^+$ and $f_I ^-$  are multiplied by the same factors as $U_I$ in \eqref{e:VI}.\\
Note that the asymptotics presented here do not even hold formally if the periods $\tau_l$ and $\tau_r$ of $b_l$ and $b_r$, respectively, are not rationally related.\\

\begin{rem}
A related problem is to study periodic mass-fluctuations $b_l = b_l(t)$ and $b_r = b_r(t)$ with $\mathcal{O}(1)$-periods and look for the limits as $t \rightarrow \infty$ of
the free boundary $p = p(t)$.  After the parabolic rescaling $x = \frac{y}{\varepsilon}$
and $t = \frac{\tau}{\varepsilon^2}$ and the usual exponential transformation the problem \eqref{e:Ueps} with $y$ and $\tau$ as independent variables is obtained, where the initial datum
$F_I(\cdot)$ is replaced by $F_I(\frac{\cdot}{\varepsilon})$. \\ 
\end{rem}

\noindent A rigorous proof of the homogenization result in the case of periodic market size fluctuations will be the subject of a future paper.\\

\section{Numerical simulations:}\label{s:numsim}
We conclude the discussion about the long time asymptotic behavior in the case of deterministic and periodic market size fluctuations by presenting numerical experiments.  \\
\noindent The numerical simulations are based on a finite difference discretization of \eqref{e:lldet} in space and a Runge-Kutta time splitting scheme. 
We use the following Strang time splitting scheme (with time step denoted by $\Delta t$ and $t_i = i \Delta t$):
\begin{enumerate}
\item Solve the ODE $F_t(x,t) = \dot{b}_l(t) F^+(x,t) - \dot{b}_r(t) F^-(t)$ explicitly in time on the interval $t \in [t_i, t_i+0.5 \Delta t]$.
\item Evolve $F$ according to the heat equation $F_t(x,t) = F_{xx}(x,t)$ for a full time step, that is $t \in [t_i, t_{i+1}]$, using an explicit
$4$-th order Runge Kutta time stepping scheme and a classic finite difference discretization in space.
\item Solve the ODE $F_t(x,t) = \dot{b}_l(t) F^+(x,t) - \dot{b}_r(t) F^-(t)$ explicitly in time on the interval $t \in [t_i+0.5 \Delta t, t_{i+1}]$.
\end{enumerate}
For the following simulations the computational domain is set to $[-50,50]$ (where $p_{\max} = 50$ corresponds to the scaled maximum price) and is split into $10^4$ equidistant subintervals if not stated otherwise.
Then \eqref{e:lldet} is solved with the following nonlinear boundary conditions
\begin{align*}
F_x(-p_{\max},t) = F_x(-p_{\max}+a,t) \quad \text{ and } \quad F_x(p_{\max},t) = F_x(p_{\max}-a,t),
\end{align*}
which correspond to homogeneous Neumann boundary conditions in the original Lasry and Lions model. Note that we choose the computational domain sufficiently large, to ensure that the price dynamics
are not influenced by the boundary conditions for a large time. However it is not possible to neglect the influence on the numerical simulations as we shall illustrate
in the first example.

\subsection{Price evolution of the classic Lasry and Lions model} In the first example we compare the theoretical long time asymptotic behavior of the Lasry and Lions model \eqref{e:lasrylions}
with the simulation results.  Note that \eqref{e:lasrylions} corresponds to the simulation of \eqref{e:F} with $\dot{b}_l(t) = \dot{b}_r(t) = 0$, that is the simulations of the heat equation
with an initial datum of the form $F_I(x) = \mathbb{1}_{x < 0} - 0.95 \times \mathbb{1}_{x > 0}$. 
Figure \ref{f:ll} illustrates the behavior of the free boundary for the Lasry and Lions model \eqref{e:lasrylions} as a function of time. The relative
error $\eta = \lvert \frac{c \sqrt{t} - p_{approx}(t)}{c \sqrt{t}}\rvert$ decays fast in a good agreement of the numerically computed price $p_{approx}$ and the theoretical results. In fact the numerical
results indicate that the relative error decays like $t^{-\frac{1}{2}}$, see Figure \ref{f:ll}b).\\
\begin{figure}
\begin{center}
\subfigure[Evolution of the price]{\includegraphics[width=0.4\textwidth]{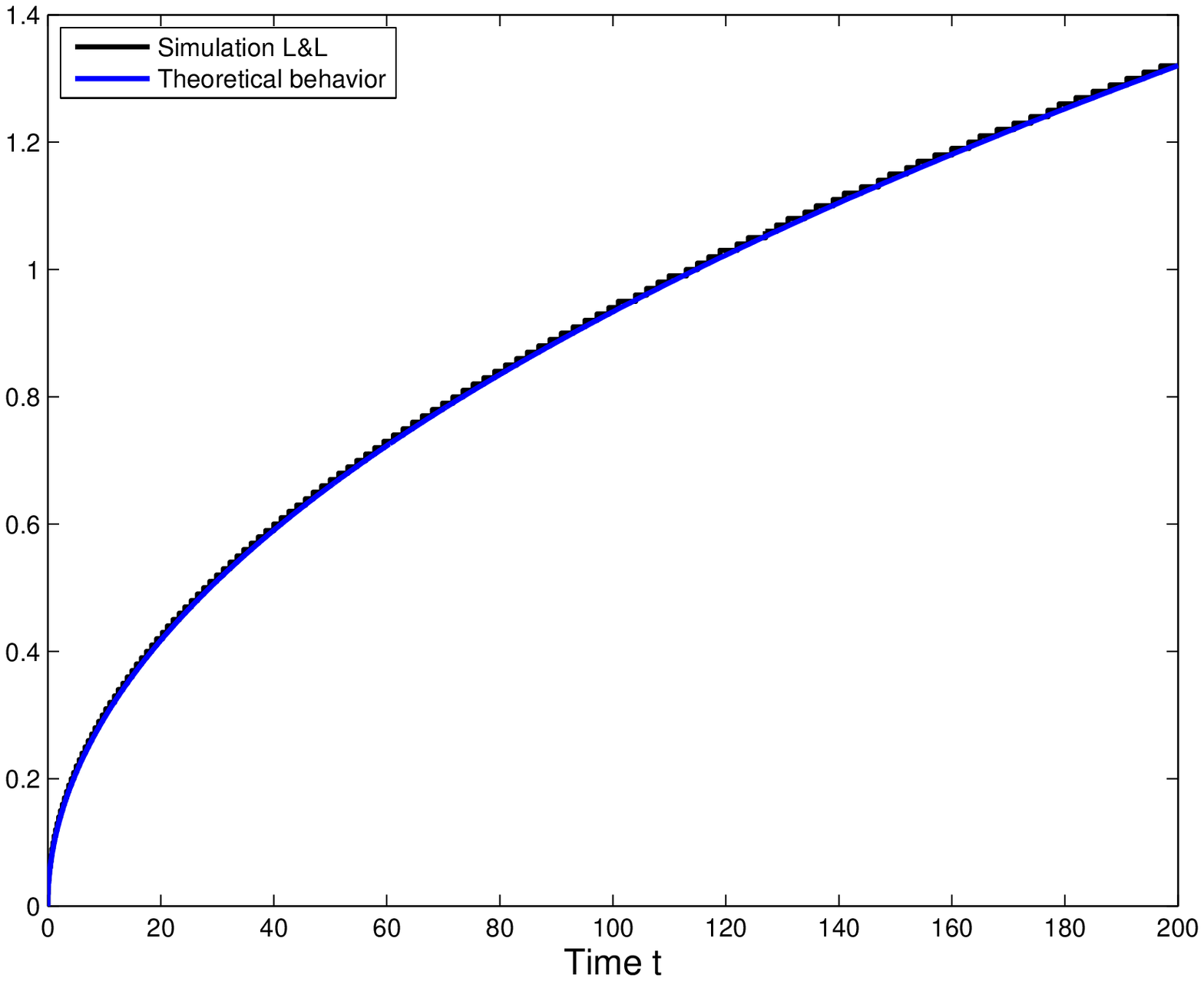}}\hspace*{0.25cm}
\subfigure[Relative error $\eta$]{\includegraphics[width=0.4\textwidth]{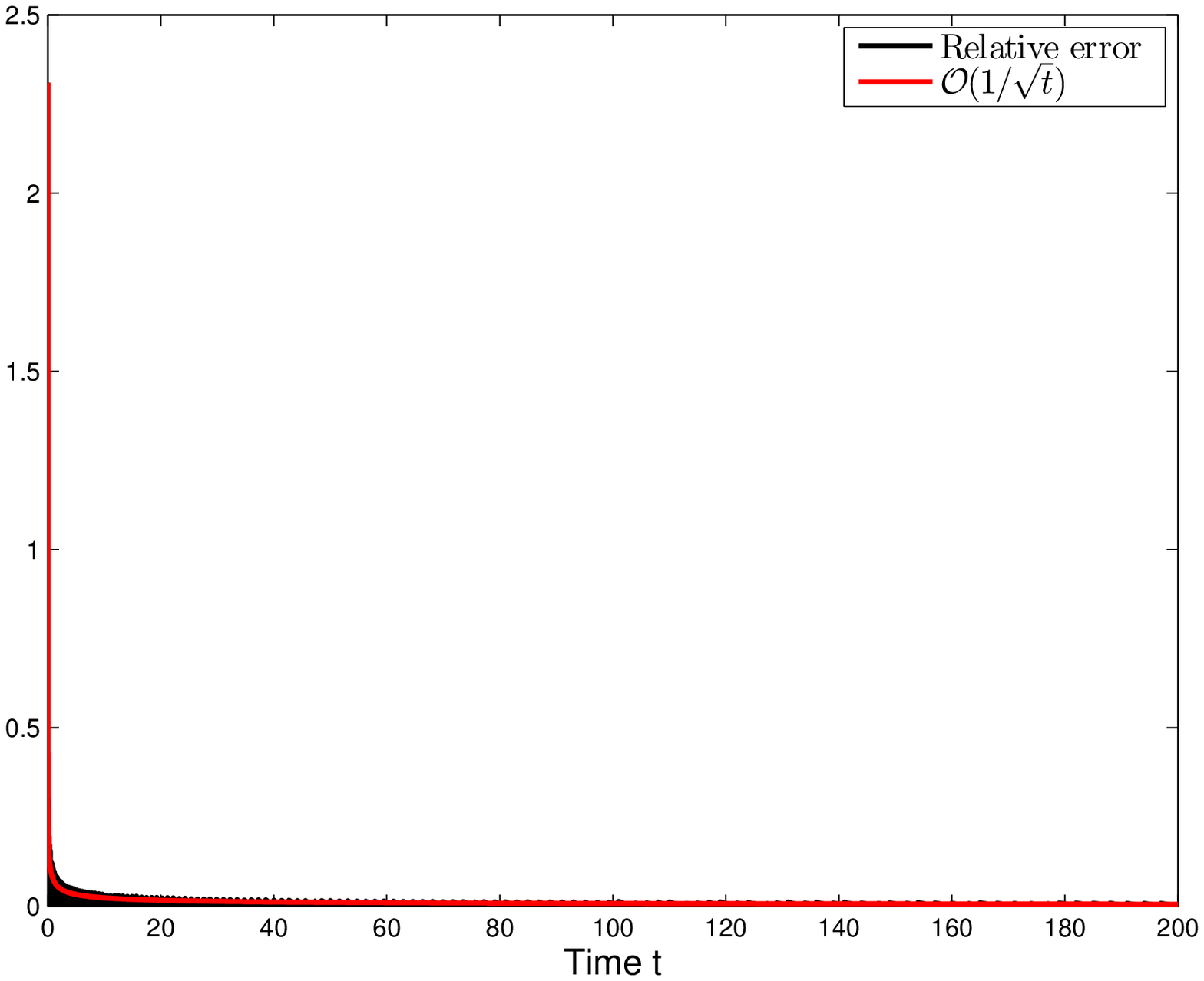}}
\caption{Comparison of the price evolution for the classic L\&L model \eqref{e:lasrylions}.}\label{f:ll}
\end{center}
\end{figure}
The next example illustrates the influence of the boundary on the price dynamics. Note that the stationary price $p_{\infty}$ in the classic Lasry and Lions model on a bounded domain $(-p_{\max}, p_{\max})$
with Neumann boundary conditions is given by
\begin{align}\label{e:pinfty}
  p_{\infty} = \frac{2 M^l p_{\max} - a (M^l-M^r)}{2(M^l+M^r)} - \frac{p_{\max}}{2}.
\end{align}
We choose an initial datum $f_I$ with equal masses of buyers and vendors, i.e.
\begin{align*}
f_I(x) = 
\begin{cases}
1  &\text{ for } -5 \leq x \leq -1\\
-x &\text{ for } -1 < x  \leq 2 \\
2 &\text{ for } 2 < x \leq 3.25 \\
0 &\text{ otherwise.}
\end{cases}
\end{align*}
 While the asymptotic analysis on the unbounded domain postulates the
convergence towards the constant $\beta_2$ given by \eqref{e:beta2}, the stationary price on the bounded domain, given by \eqref{e:pinfty}, corresponds to $p_{\infty} = 0$. Figure    
\ref{f:llequmass} clearly illustrates the change of the price dynamics for differently sized domains $\Omega = [-p_{\max}, p_{\max}]$. The black line at $p=-0.3475$ corresponds to
the predicted analytic constant - we observe that the prices initially converge to this value, but then reverse their dynamics towards $p_{\infty} = 0$ due to the influence of the boundary
conditions. The smaller the computational domain the earlier the reversal happens. 
\begin{figure}[h!]
\begin{center}
\includegraphics[width=0.5\textwidth]{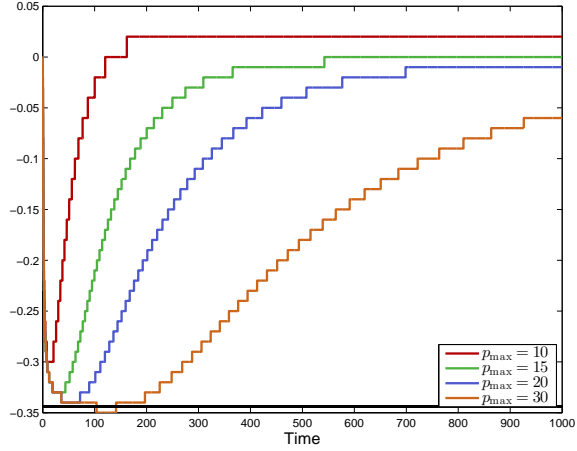}
\caption{Evolution of the free boundary in the classic Lasry and Lions model in the case of equal market size distribution for different computational domains $\Omega = [-p_{\max}, p_{\max}]$.}\label{f:llequmass}
\end{center}
\end{figure}

\subsection{Price evolution for stabilizing market size fluctuations}
Next we compare the simulated price dynamics for different choices of $b_l$ and $b_r$ with the theoretical results. We consider the three cases:
\begin{align}\label{e:detfluc}
\begin{split}
&\text{linear: } ~b_l(t) = c_l \frac{ t}{t+1},~\quad \text{ exponential: } b_l(t) = c_l(1-\exp(-t^2))\\
&\text{quadratic: } ~b_l(t) = c_l \frac{t^2}{t^2+1},
\end{split}
\end{align}
with $c_l = 1.12345$. The fluctuations of the vendor market size distribution have the respectively same form with $c_r = 1.0$. Hence $b_l^\infty = c_l$ and $b_r^\infty = 1$. We start with simulations in the case of different initial masses of buyers and vendors. The initial datum is given by 
$F_I(x) = \mathbb{1}_{x < 0} - 0.9 \times \mathbb{1}_{x > 0}$. 
Figure \ref{f:deterministic} compares the theoretical long time asymptotics based on the results of Theorem \ref{t:thm1} and the numerical simulations. We observe that
the convergence rate of the functions $b_l$ and $b_r$ towards their asymptotic value determines the constant $o(1)$ - the faster the convergence the closer the
behavior with respect to the theoretical predictions. The relative errors for the respective cases in Figure \ref{f:deterministic}b) illustrate the fast convergence of the free boundary, they decay
 like $t^{-\frac{1}{2}}$ (the inset in Figure \ref{f:deterministic}(b) corresponds to the enlarged view of the price dynamics for small times).\\
\begin{figure}[h!]
\begin{center}
\subfigure[Evolution of the price.]{\includegraphics[width=0.45\textwidth]{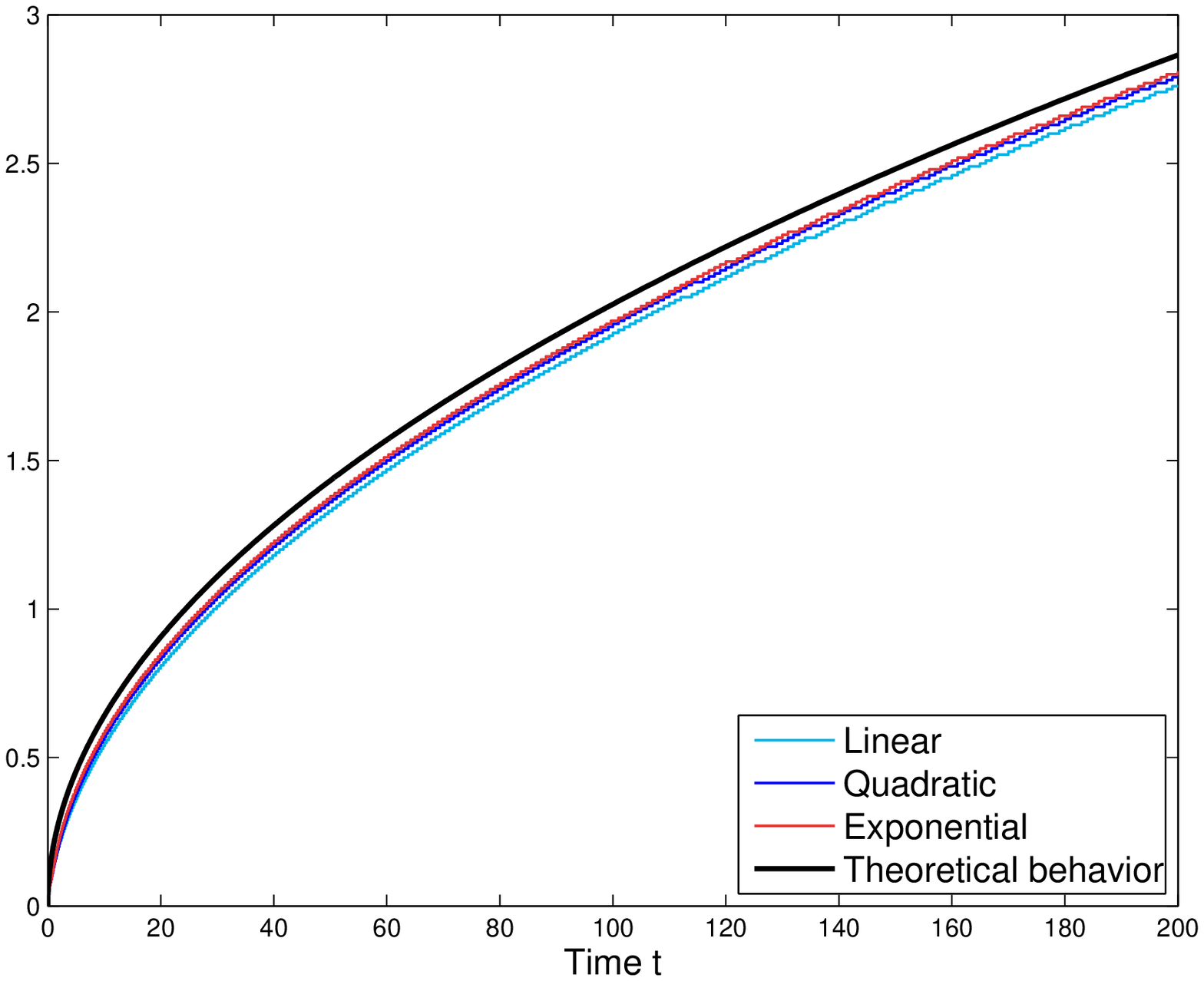}}
\subfigure[Relative error]{\includegraphics[width=0.45\textwidth]{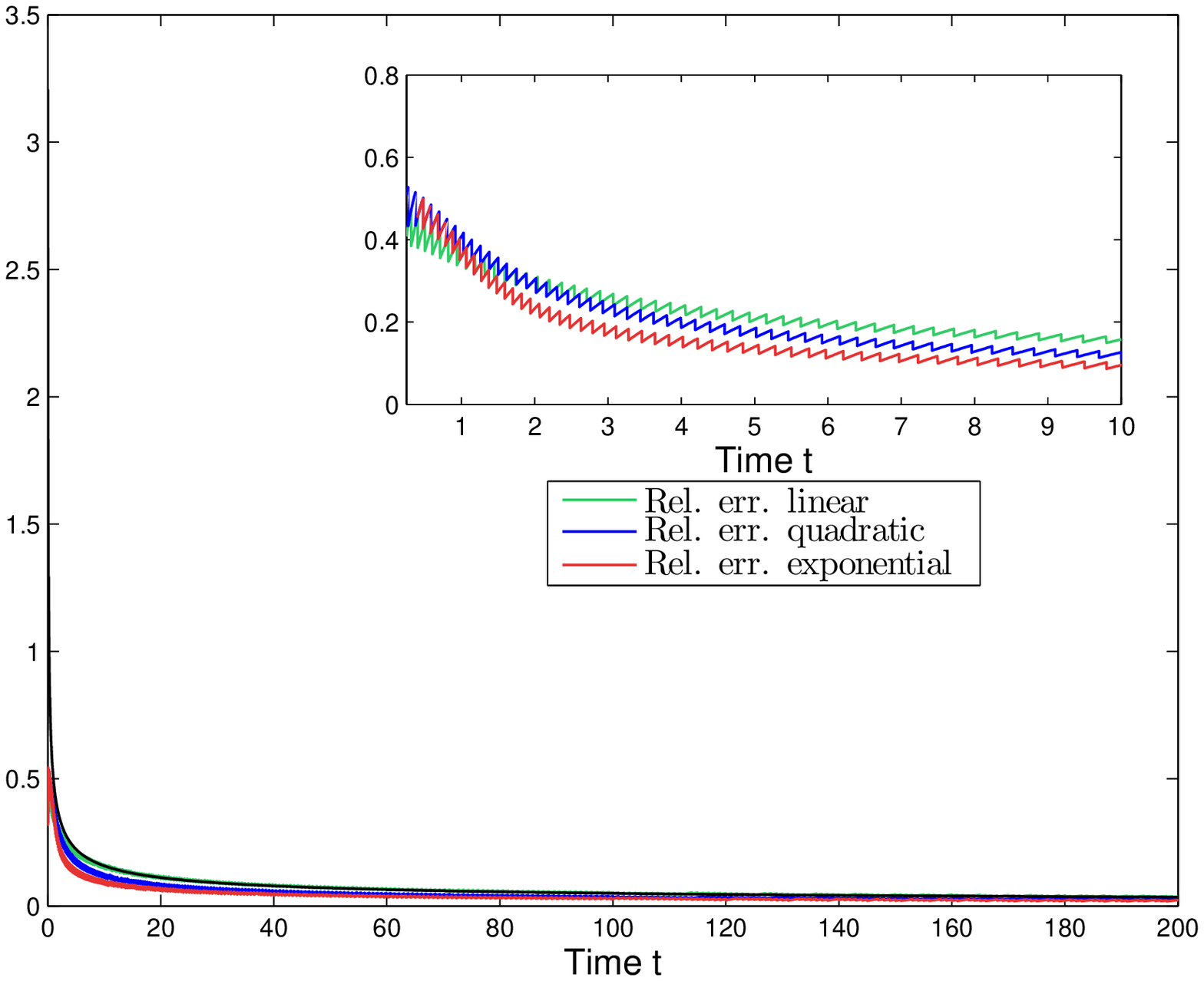}}
\caption{Simulations in the case of deterministic fluctuations \eqref{e:detfluc} with $\alpha^l \neq \alpha^r$.}\label{f:deterministic}
\end{center}
\end{figure}

\noindent Next we consider the case $\alpha^l = \alpha^r$, that is the number of buyers and vendors converge to the same long-time limit. Theorem \ref{t:thm2} states that the long time asymptotic behavior is driven by the first order moments of the initial data. We choose the following initial datum:
\begin{align*}
F_I(x) &=
\begin{cases}
40 &\text{ if } x < -2,\\
-20 x &\text{ if  } -2 \leq x \leq 3,\\
-60 & \text{ if } x > 3.
\end{cases}
\end{align*}
Furthermore we set 
\begin{align*}
b_l(t) = c_l(1-e^{-5t}) \quad \text{ and } \quad b_r(t) = c_r(1-e^{-5t})  
\end{align*}
with $c_l = \log(1.1) + \log(1.5)$ and $c_r = \log(1.1)$. Note that the functions $b_l$ and $b_r$ satisfy Assumption \ref{a:flucdec2} and that the parameter choices imply $\alpha^l = \frac{l^{\infty} M^l}{a} =  \frac{r^{\infty} M^r}{a} = \alpha^r$. The evolution of the price is depicted in Figure \ref{f:deterministic2}. We observe the expected long time behavior towards the theoretically predicted constant $\beta_2$, but the decay is different. It decreases like $\exp^{-5t}$, illustrated in Figure \ref{f:deterministic2}b), and not like $t^{-\frac{1}{2}}$ as in the case $\alpha^l \neq \alpha^r$.
\begin{figure}
\begin{center}
\subfigure[Evolution of the price]{\includegraphics[width=0.4\textwidth]{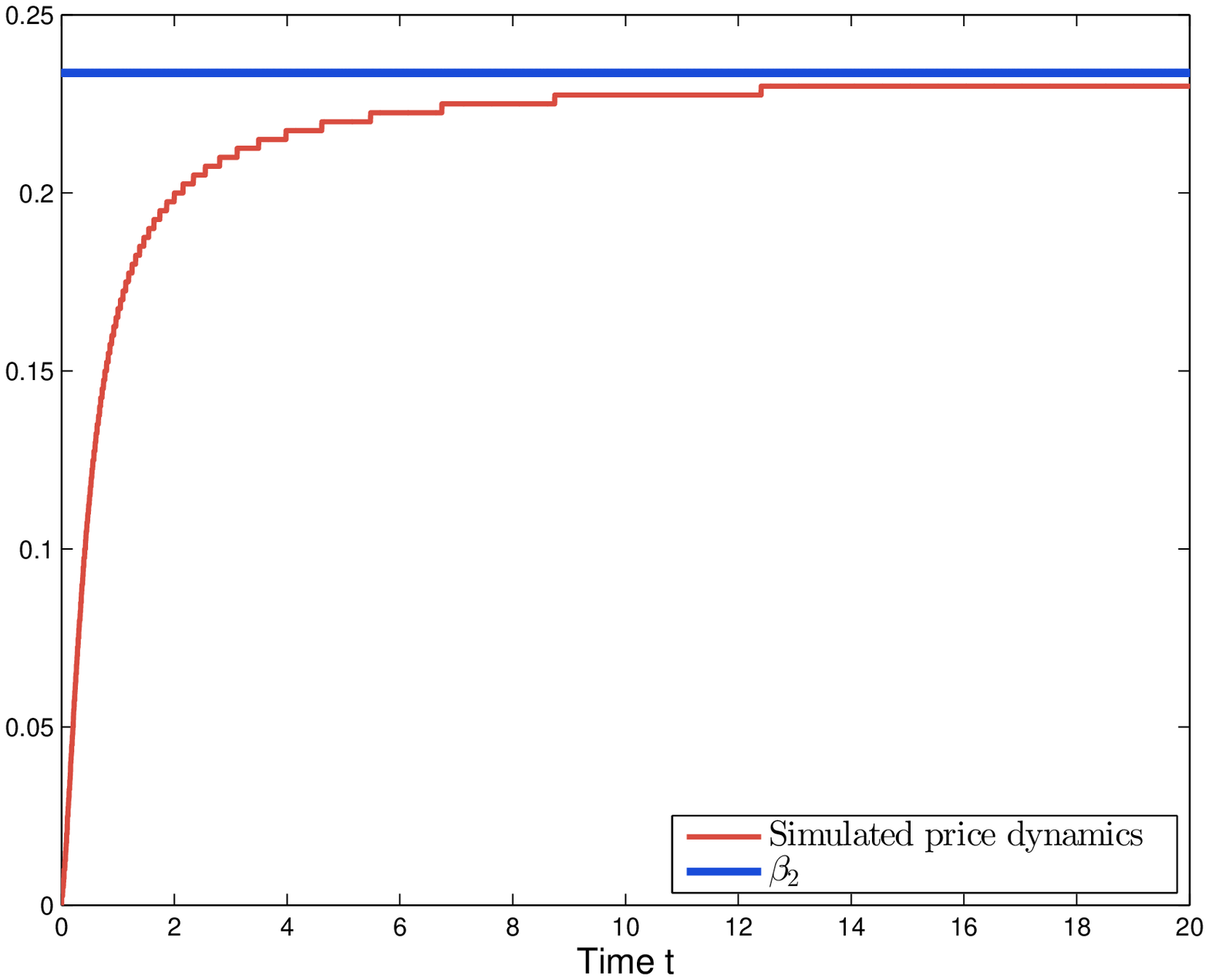}}\hspace*{0.5cm}
\subfigure[Relative error]{\includegraphics[width=0.4\textwidth]{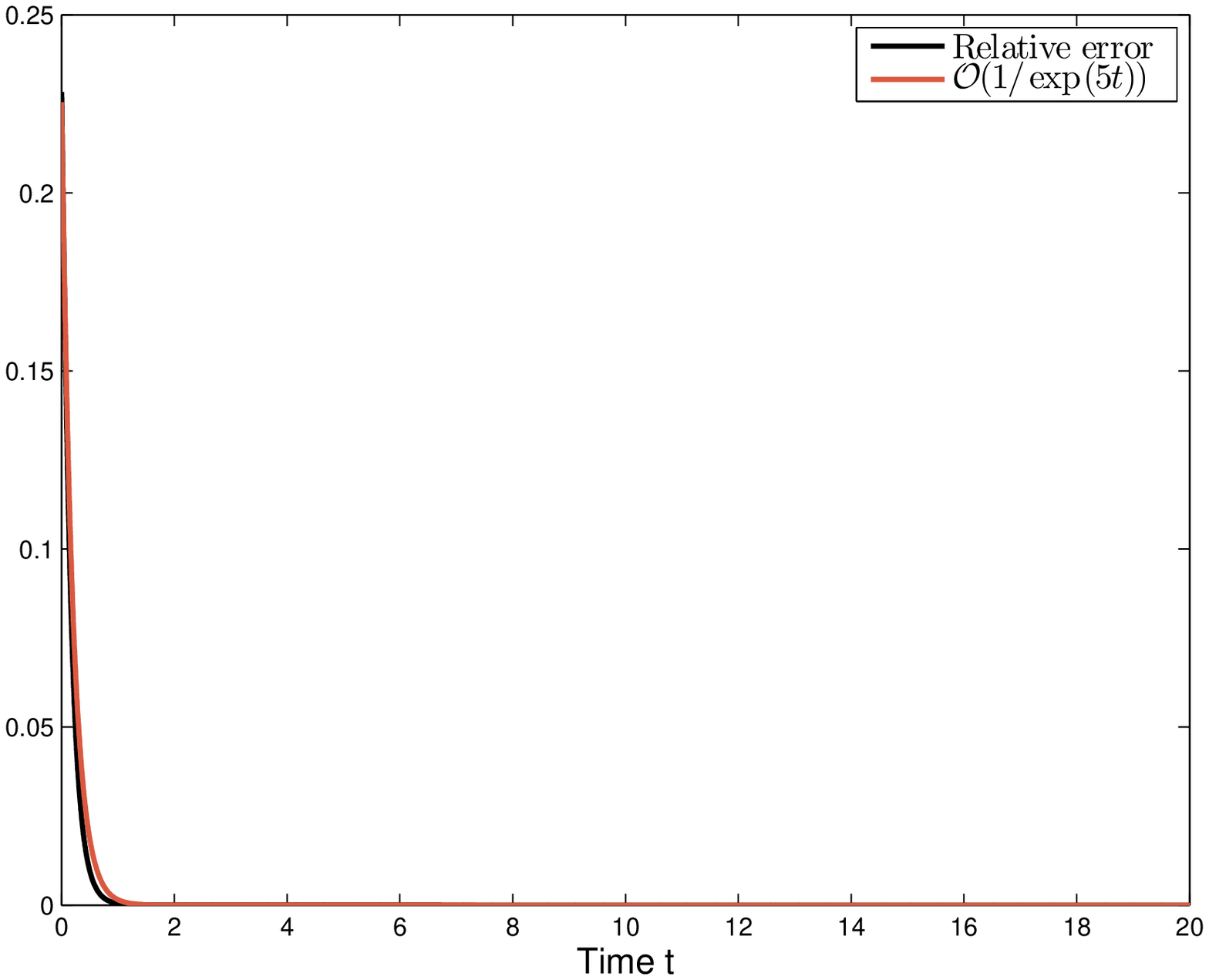}}
\caption{Simulations in the case of deterministic fluctuations  \eqref{e:detfluc} with $\alpha^l = \alpha^r$.}\label{f:deterministic2}
\end{center}
\end{figure}

\subsection{ Price evolution in the case of periodic market size fluctuations} Next we consider periodic fluctuations in the buyer and vendor market sizes. First
we choose fluctuations of the form 
\begin{align}\label{e:per1}
b_l(t) = \cos (2 \pi f_l t)~~ \text{ and }~~ b_r(t) = \cos (2 \pi f_r t).
\end{align}
Figure \ref{f:periodic} illustrates  the behavior of the free boundary on the time interval $t \in [0,1]$ for different values of $f_l$ and $f_r$, when the ratio $f_l\backslash f_r$ is rational. We observe that the maximum amplitude of the computed price from the theoretical predictions decreases as $\varepsilon \rightarrow 0$, an 
observation also confirmed in Table \ref{t:periodic}.\\

\begin{figure}[h!]
\centering
\includegraphics[width=0.45\textwidth]{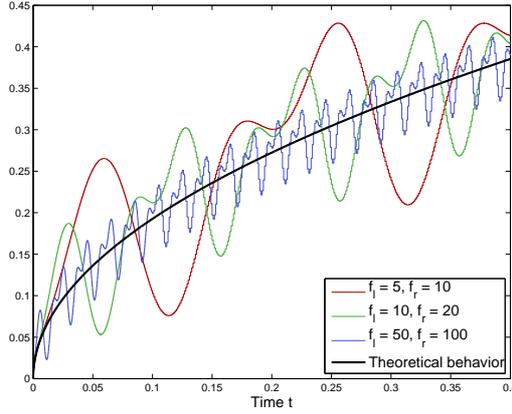}
\caption{Simulations in the case of periodic fluctuations \eqref{e:per1} with zero mean.}\label{f:periodic}
\end{figure}
\begin{table}[h!]
\centering
\begin{tabular}{|c|c|c|c|c|}
\hline
$f_l$/$f_r$ & 5/10 & 10/20 & 20/40 & 50/100 \\
\hline
amplitude & 0.1389 & 0.1018 & 0.0777 & 0.0538 \\
\hline
\end{tabular}
\caption{Maximum deviation of the free boundary from the theoretical prediction due to periodic mass fluctuations of the form \eqref{e:per1} in case of a rational fraction $\frac{f_l}{f_r}$. }\label{t:periodic}
\end{table}
Note that the functions $b_l = b_l(t)$ and $b_r = b_r(t)$ are one-periodic with zero mean in the first example and $b_l(0) = b_r(0) = 1$. To illustrate the impact of mass fluctuations with
non-zero mean we choose functions $b_l$ and $b_r$ of the form:
\begin{align}\label{e:per2}
b_l(t) = \cos (\pi t)^2~~ \text{ and }~~ b_r(t) = \cos (2 \pi t).
\end{align}
For this setting the corresponding evolution is illustrated in Figure \ref{f:periodic2} (using the same simulation parameters as in the previous example).
\begin{figure}[h!]
\begin{center}
\subfigure[Evolution of the price]{\includegraphics[width=0.45\textwidth]{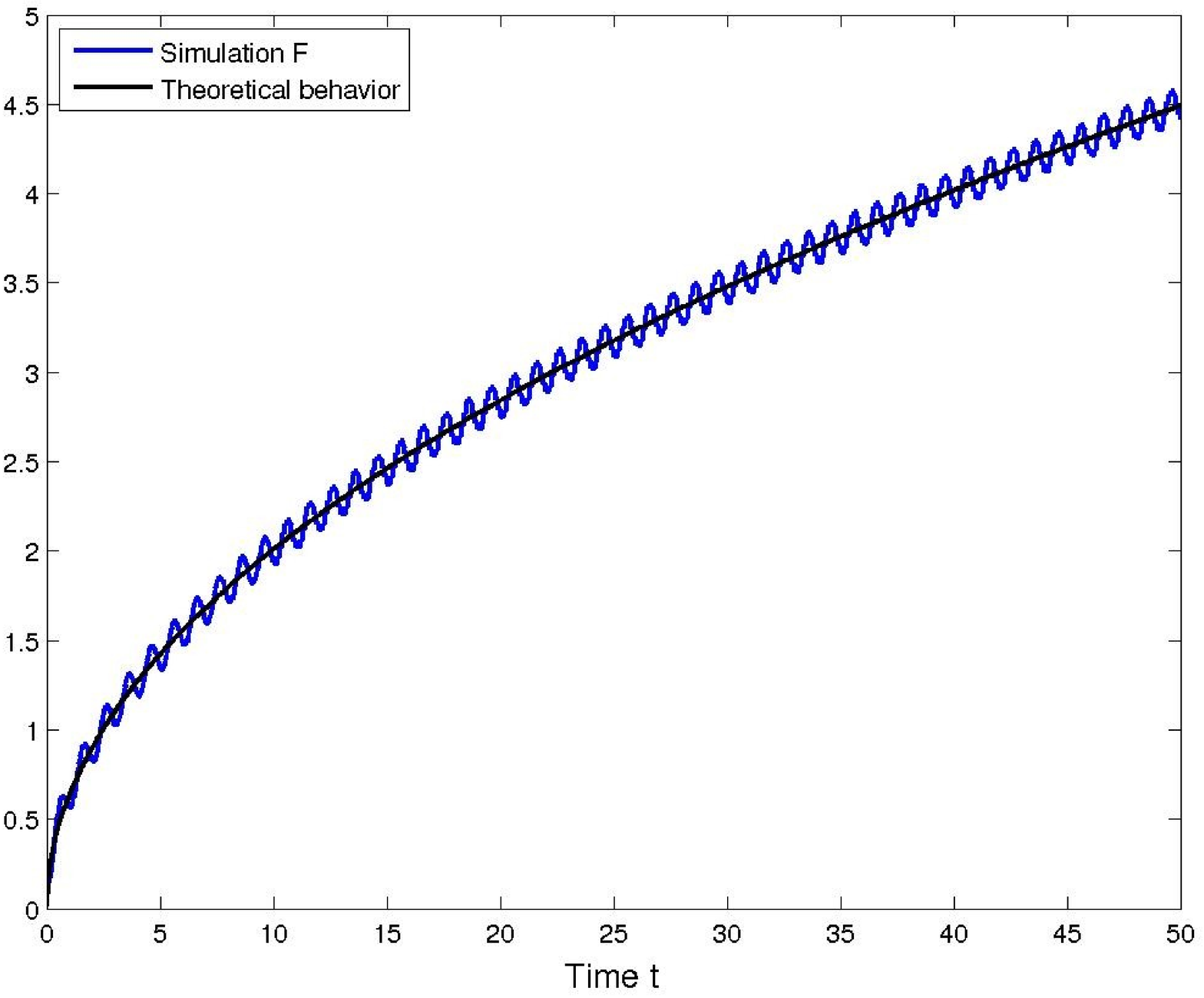}}
\subfigure[Relative error]{\includegraphics[width=0.45\textwidth]{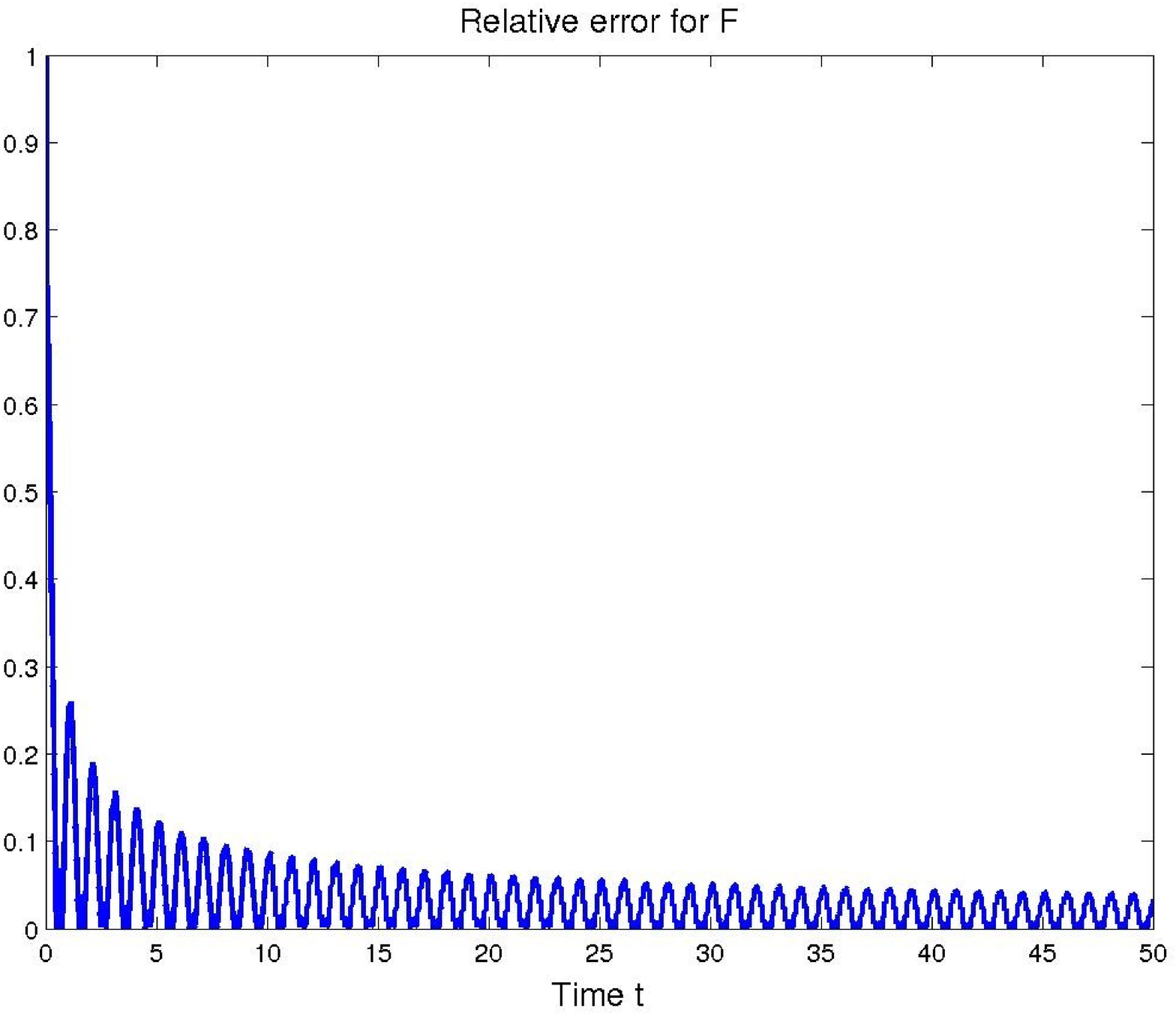}}
\caption{Simulations in the case of periodic fluctuations \eqref{e:per2} with non-zero mean.}\label{f:periodic2}
\end{center}
\end{figure}
Again we observe a very good agreement with the predicted theoretical price dynamics (based on the modified initial datum \eqref{e:VI}), even though they are based on
formal arguments.\\ 

\noindent Finally we study the effect of periodic mass fluctuations on the free boundary. We choose an initial datum with equal mass, that is $F_I(x) = \mathbb{1}_{x<0} - \mathbb{1}_{x>0}$ 
to clearly distinguish the influence of the market size fluctuations from the price dynamics due to the market imbalance of buyers and vendors. The functions $b_l$ and 
$b_r$ are set to
\begin{align}
b_l(t) = \cos (\pi f t)^2 ~~ \text{ and } ~~ b_r(t) = \sin(2 \pi f t).
\end{align}
 Table \ref{t:pricefluc} states the computed amplitude and frequency for the different values of $f$. Note that mass oscillations of order $\mathcal{O}(1)$ amplitude result in
$\mathcal{O}(\varepsilon)$-amplitude fluctuations in the free boundary and that we have 'numerical' strong convergence of the free boundary.
\begin{table}[h!]
\begin{center}
\begin{tabular}{|c | c | c | c | c |}
\hline
$f$ & 5 & 10 & 20 & 100 \\  
\hline
amplitude &  0.1415 & 0.10446 & 0.075766 & 0.03982  \\
frequency & 4.8832 & 9.766 & 20.142 & 99.9456\\
\hline
\end{tabular}
\caption{Frequency and amplitude of the price fluctuations for different values of $f$.}
\label{t:pricefluc}
\end{center}

\end{table}

\subsection{Almost periodic market size fluctuations}\label{s:almostpernum}

We would like to argue numerically that the formal asymptotic for the periodic case, i.e. that the limiting dynamics of the price correspond to the dynamics of the original Lasry-Lions model with a changed initial buyer-vendor distribution \eqref{e:VI}, holds also in the case of Besicovitch almost periodic market size fluctuations, cf. \cite{B55}. Note that the space of Besicovitch almost periodic
functions  corresponds to the closure of the trigonometric polynomials under the seminorm
\begin{align*}
\lVert f \rVert_{B,p} = \limsup_{x\rightarrow \infty} \left(\frac{1}{2x} \int_{-x}^x \lvert f(s)\rvert^p ds \right)^{\frac{1}{p}}.
\end{align*} 
These functions have an expansion of the form $\sum_n a_n e^{i \lambda_n t}$ with $\sum a_n^2$ finite and $\lambda_n \in \mathbb{R}$; their mean value is defined by
\begin{align}\label{e:genmean}
\langle f \rangle := \lim_{T\rightarrow \infty}\frac{1}{T} \int_0^T f(s) ds.
\end{align}
 We set
 \begin{align*}
b_l(t) =  \sin(2 \pi f_l t) + \sin(2 \pi f_l t \sqrt(2)) \text{ and } b_r(t) = \cos(2 \pi f_r t)
\end{align*}
with integer frequencies $f_l$ and $f_r$ and choose an initial datum of the form $F_I(x) = \mathbb{1}_{x < 0} - \frac{1}{2} \mathbb{1}_{x \geq 0}$.
Figure \ref{f:quasiperiod} illustrates the price dynamics for the time interval $t \in [0,1]$ and the frequencies $f_l = f_r = 1, 5, 10, 20, 40$. The black line 
corresponds to the theoretical prediction computed as in the periodic case, using the parameters $ \langle b_l \rangle = \langle b_r \rangle  = 0$, $b_l(0) = 0$ and $b_r(0) = 1$.
As already explained we observe convergence of the price towards the formal theoretical prediction in the periodic case, see \eqref{e:VI}. The numerical results indicate also the validity of the formal asymptotics in the almost periodic case.\\
\begin{figure}
\centering
\subfigure[Price dynamics]{\includegraphics[width = 0.4 \textwidth]{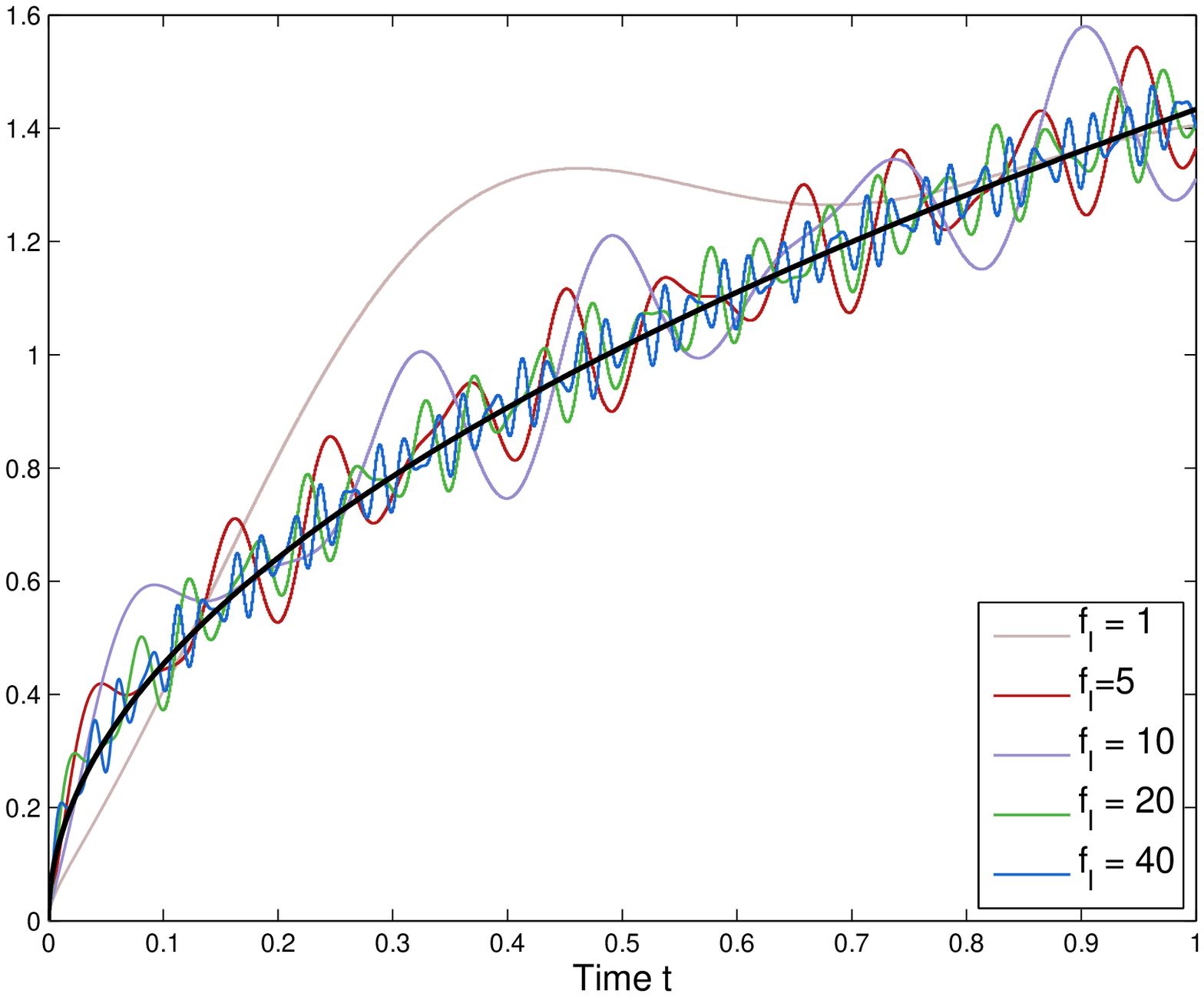}}\hspace*{0.4cm}
\subfigure[Approximation of the price dynamics for $f_l=f_r = 10$.]{\includegraphics[width = 0.4 \textwidth]{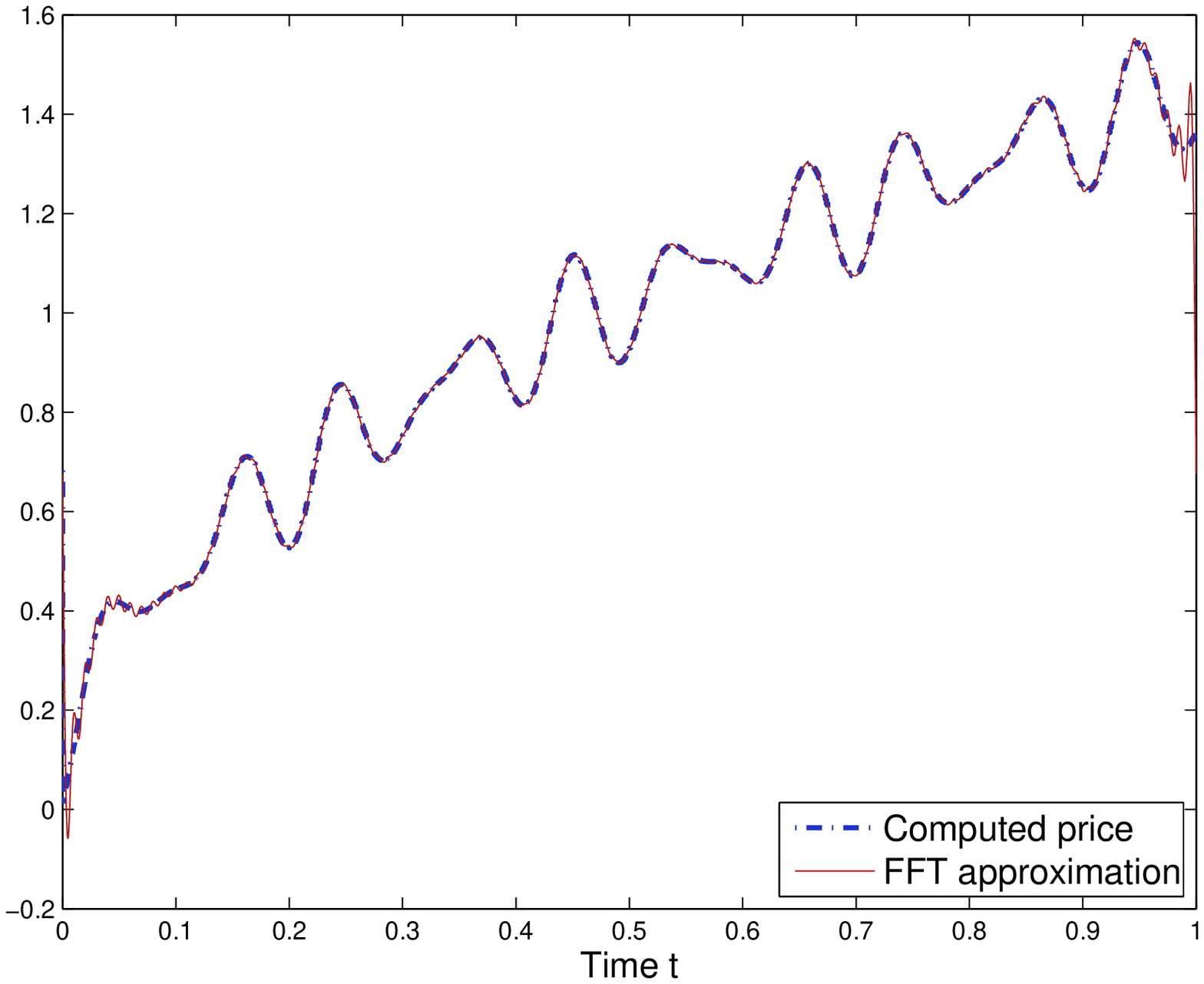}}
\caption{Price dynamics in case of almost-periodic market size fluctuations.}\label{f:quasiperiod}
\end{figure}
We would like to remark that it is difficult to determine if the discrete price dynamics are almost-periodic functions. Figure \ref{f:quasiperiod} shows the approximation
of the price in the case $f_l = f_r = 10$, which is based on the leading $100$ FFT coefficients of the discrete price. 

\subsection{Price formation with  stochastic market size fluctuations}\label{s:stochastic}

\noindent Finally we consider the Lasry and Lions model \eqref{e:lasrylions} with stochastic fluctuations in the numbers of buyers and vendors, which we model by two independent Brownian motions $B_l = B_l(t)$ and $B_r = B_r(t)$ replacing the deterministic functions $b_l = b_l(t)$ and $b_r = b_r(t)$ in Section \ref{s:deterministic}. Since under sufficient regularity conditions (such that Wong-Zakai type theorems for SPDEs hold, see ~\cite{nak:2004}) stochastic partial differential equations (SPDE) involving the Stratonovich integral are obtained in the limit (in probability) when Brownian motions are (reasonably) approximated by smooth functions. We shall use the  Stratonovich version of the Lasry and Lions model following the common convention that $ \circ B $ denotes Stratonovich integration against $B$. \\
\noindent Notice that the Stratonovich formulation is well defined when the Ito integral is well defined \emph{and} the integrand is a semimartingale. In our case this can be guaranteed due to the special structure of the SPDE's characteristics if sufficiently regular initial values are chosen. Also the Stratonovich correction term has a simple expression, since $ F \mapsto \lvert F \rvert $ and $F \mapsto F$ lead at least formally to the correction term $ 1/2 F $. In the sequel we do not enter too much into the stochastic details but assume enough regularity for the initial values such that our considerations hold true.\\

\noindent Here we discuss the well posedness of the equation, but leave a more refined analysis, that is of long term limits, for future research. As a matter of fact we show that the influence of multiplicative stochastic market size fluctuations on the price evolution is relatively weak.

\noindent Using the same transformation \eqref{e:trans} as in the beginning of Section \ref{s:deterministic} we obtain the stochastic equivalent of \eqref{e:lldet}, in particular a heat equation with a multiplicative Lipschitz non-linearity in the stochastic term:
\begin{align}\label{e:stochF1_stratonovich}
\begin{split}
d F (x,t) = F_{xx}(x,t) dt &+ \frac{1}{2}F(x,t) \circ (\sigma_l dB_l + \sigma_r dB_r)  {}\\
&+\frac{1}{2}\lvert F(x,t) \rvert \circ (\sigma_l dB_l - \sigma_r dB_r).
\end{split}
\end{align}
We apply the transformation
\begin{align}
  V(x,t) = F(x,t) e^{-\frac{\sigma_l B_l + \sigma_r B_r}{2}},
\end{align}
and obtain using the Stratonovich chain rule that
\begin{align}
  dV(x,t) = V_{xx}(x,t) + \frac{1}{2} \lvert V(x,t) \rvert \circ (\sigma_l dB_l - \sigma_r dB_r).
\end{align}
The corresponding Ito equation is given by
\begin{align}\label{e:stochF1_ito}
\begin{split}
d V (x,t) = V_{xx}(x,t) dt &+ \frac{\sigma_l^2+\sigma_r^2}{8} V(x,t) dt   {}\\
&+\frac{1}{2}\lvert V(x,t) \rvert (\sigma_l dB_l + \sigma_r dB_r).
\end{split}
\end{align}
Another exponential transformation
\begin{align*}
U(x,t) = V(x,t) e^{-\frac{(\sigma_l^2 + \sigma_r^2)}{8}t}
\end{align*}
yields, since $B_l(0) = B_r(0) = 0$,
\begin{subequations}\label{e:LLstoch_reform}
\begin{align}
&d U (x,t)=U_{xx}(x,t) dt + \frac{1}{2}\lvert U(x,t) \rvert (\sigma_l dB_l - \sigma_r dB_r) \, \text{ in } \mathbb{R} \times \mathbb{R}_+ ,\\
&U(x,t=0) = F_I(x).
\end{align}
\end{subequations}
This equation can be analyzed by the DaPrato-Zabczyk methodology (see \cite{dapzab:2014}), which roughly speaking says that on every Hilbert space $\mathcal{H}$ of real valued functions on $\mathbb R $ where 
\begin{enumerate}[label=(\arabic*)]
\item the map $F \rightarrow \lvert F \rvert$ is globally Lipschitz on $\mathcal{H}$, and
\item the one-dimensional Laplacian $ F \mapsto F_{xx} $ generates a strongly continuous semigroup $S$ (appropriately closed) on $\mathcal{H}$,
\end{enumerate}
there exists a unique mild solution of the SPDE. This process $ {(U(t))}_{t \geq 0} $, indexed in time $t$ and taking values in $\mathcal{H}$, satisfies, for every initial value $ F(0) \in \mathcal{H}$, the Duhamel formulation
\begin{align}\label{LL-stochastic}
\begin{split}
U(t)  = S_t U(0)  + \frac{1}{2} \int_0^t S_{t-s} \lvert U (s) \rvert (\sigma_l d B_l(s) - \sigma_r dB_r (s)). 
\end{split}
\end{align}
As usual we suppress the space (price) variable in this notation. The uniqueness of the solution of \eqref{e:LLstoch_reform} is understood in the Banach space of cadlag processes with  norm 
\[
U \mapsto \sqrt{E \big [\sup_{0 \leq s \leq T} {\| U(s) \|}^2 \big ]} \, .
\]
Note that the form is not restricted to independent Brownian motions. In fact \eqref{LL-stochastic} can be solved for any two dimensional semi-martingale driving process $ (S_l, S_r) $, in particular also for L\'evy processes; see Protter ~\cite{pro2005} and, Peszat and Zabczyk \cite{PZ2007} for more information.\\

 \noindent From \eqref{LL-stochastic} we conclude that $U$ is given by the Feyman-Kac formula
\begin{align*}
&U(x,t) = E \Big[ F_I(x+\sqrt{2} W_t) \times  \exp \big(\frac{1}{2} \int_0^t \operatorname{sign} (U(x,s)) (\sigma_l dB_l(s) - \sigma_2 dB_r(s))\big) \Big].
\end{align*}
This representation formula is justified by approximating the Brownian motions by piecewise constant interpolations and an application of the classical Feynman-Kac formula.
\noindent Differentiating \eqref{LL-stochastic} with respect to $x$, we find
\begin{align}
\begin{split}
U_x(t)  = S_t U_x(0) + \frac{1}{2} \int_0^t S_{t-s} \operatorname{sign} \left(U (s)\right) U_x (s) \left(\sigma_l d B_l(s) - \sigma_r d B_r (s)\right) - t (\sigma_l^2 + \sigma_r^2)/8  \, .
\end{split}
\end{align}
which then gives the  Feynman-Kac representation namely
\begin{align*}
  & U_x(x,t)  =  E \Big[ F_{I,x}(x+\sqrt{2} W_t) \times \\
  & \quad \times\exp \big( \frac{1}{2} \int_0^t (\operatorname{sign}(U(x+\sqrt{2}W_s,s)) \sigma_l d B_l(s) -\operatorname{sign}(U(x+\sqrt{2}W_s,s)) \sigma_r d B_r(s)) - t(\sigma_l^2+\sigma_r^2)/8 \big) \Big] \, 
\end{align*}

\noindent It is an easy consequence of this representation that for monotonically decreasing non-constant initial data $F_{I}$ positivity of $U_x (t)$ holds instantaneously for $t > 0$.
 This implies that, for such initial data, there exists at most one zero $x = p(t)$ of $U(t)$ for every $t>0$, that is the free boundary cannot turn back or develop fat parts
in finite time. By smoothing the Brownian processes $B_l$ and $B_r$ in \eqref{e:LLstoch_reform} we conclude for general initial data $F_I$ that the zero level set of the solution $U$ is the graph limit 
of a sequence of functions of time. This excludes in particular that the free boundary can turn back but it does not exclude the formation of fat parts. \\

\noindent Although the main focus of this paper was the long time behavior of the free boundary we conclude with a short heuristic calculation which illustrates why we cannot expect very volatile dynamics in the case of Brownian drivers: we consider an (exponential) Euler step for a short period of time denoted by $ t >0 $. Let $S$ denote the heat semigroup and $ E_{0,t} $ the random operator defined for $ x \leq p(t) $ by 
\begin{subequations}\label{e:randop}
\begin{align}
E_{0,t} U(x) := U(x) \exp \big ( \sigma_l B_l(t) - \frac{\sigma_l^2t}{2} \big),
\end{align}
 and for $ x \geq p(t) $ by
\begin{align}
E_{0,t} U(x) := U(x) \exp \big ( \sigma_r B_r(t) - \frac{\sigma_r^2t}{2} \big).
\end{align}
\end{subequations}
 We  evaluate the concatenation of these operators for a $C^1$ function $ U $ with precisely one zero at $p$ with $ U'(p) < 0 $ for small times $t$ and for $x$ close to $p$. We have
\begin{align*}
S_t E_{0,t} U(x) = \int_{-\infty}^{\infty} \phi(y) \big( &1_{\{x+\sqrt{2t}y \leq p\}} \kappa_l U'(p)(x+\sqrt{2t} y -p) + \\
 +  &1_{\{x+\sqrt{2t}y \geq p\}} \kappa_r U'(p)(x+\sqrt{2t}y-p) \big) dy \, ,
\end{align*}
where $ \phi $ is the standard normal density and
\[
\kappa_{l/r}=\kappa_{l/r}(t) = \exp \big (\sigma_{l/r}  B_{l/r}(t) - \frac{\sigma_{l/r}^2 t}{2} \big) \, .
\]
For small times $t$ we obtain
\begin{align*}
S_t E_{0,t} U(x) & = U'(p) \kappa_l (x-p) \Phi(x - p) + U'(p) \kappa_r (x-p) (1-\Phi(x - p)) \\
& + U'(p) (-\kappa_l+\kappa_r) \sqrt{2t} \phi(x-p) \, ,
\end{align*}
which implies that for small times $p(t)$ is of the form:
\begin{align}\label{e:shortasymptotic}
p(t) \approx p+\frac{2 (\kappa_r-\kappa_l)}{\sqrt{\pi}(\kappa_r+\kappa_l)}\sqrt{t} +\mathcal{O}(t) \, .
\end{align}
This formula implies that we do actually observe $\sqrt{t} $ asymptotics weighted with factors stemming from $\kappa_r-\kappa_l$ for small times as well. The asymptotics are also of order of magnitude $ \sqrt{t} $ due to the Brownian increments in the exponent. This implies that the increments of the price process $p(t) - p(0) $  are $ \mathcal{O}(t) $ even in the presence of Brownian drivers, and therefore of finite total variation for $ t \mapsto p(t)$. Hence the price process does not admit a martingale part, which implies that even proportional fluctuations of buyers and vendors densities by independent Brownian motions  do not lead to price curves as observed in realistic markets.\\
On the other hand the short time asymptotics give a hint how to introduce ``more'' randomness to obtain a martingale price process $ p $. This is that the densities of the buyer and vendor must be perturbed on small intervals of length $t$ by multiplicative noises of $ \mathcal{O}(1) $. This, however, does not lead to the type of SPDE \eqref{LL-stochastic} originally proposed in this section.

\paragraph*{Numerical simulations}
In the first example we would like to confirm the small time asymptotic behavior of the free boundary $p(t)$ given by \eqref{e:shortasymptotic}. We start with an initial datum of the form
\begin{align*}
F_I(x) = -x, 
\end{align*}
and set $\sigma_r = \sigma_l = 1$. We use the domain $\Omega$ to $[-10,10]$ and run the simulation on the time interval $t \in [0,1]$. The spatial
discretization corresponds to equidistant intervals of size $h = 10^{-3}$, while the time steps are set to $\Delta t = 10^{-5}$. At time $t=0.5$
we choose two independent normally distributed random numbers 
$\Delta B^1 $, $ \Delta B^2$ with standard
deviation $ \sqrt{\Delta t} $ and multiply the distribution left of the price with $ \exp(\Delta B^1 - \Delta t/2) $ and the right of
the price with  $ \exp(\Delta B^2 - \Delta t/2) $. Figure \ref{f:smalltime} illustrates the small time asymptotics - the initial
price up to time $t=0.5$ equals to zero due to the same total number of buyers and vendors. Then the stochastic fluctuations 
in the buyer and vendor density initiate a new price dynamic, which confirms the estimated behavior
given by \eqref{e:shortasymptotic}.
\begin{figure}
\begin{center}
\includegraphics[width=0.5\textwidth]{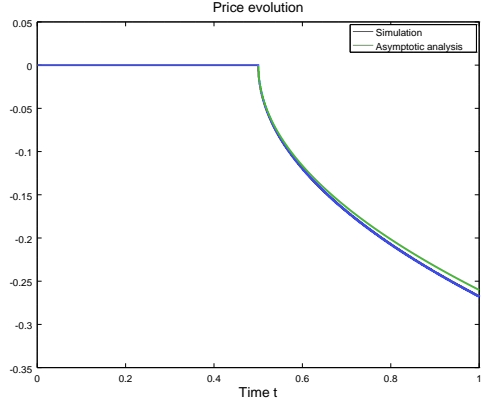}
\caption{Short-time asymptotics in case of stochastic market size fluctuations}\label{f:smalltime}
\end{center}
\end{figure}
\noindent We conclude with an example illustrating the necessity to add ``more'' randomness to the model. We start with an initial datum of the form $F_I(x) = \mathbb{1}_{x <0} - \mathbb{1}_{x> 0}$ and
time steps of $\Delta t=10^{-5}$. At time $t = 0.25,~ 0.5$ and $t=0.75$ we multiply the right and left distribution by two independent normally distributed random numbers 
$\Delta B^1$ and $\Delta B^2$ with standard deviation $\sqrt{\Delta t_{slow}}$, where $\Delta t_{slow}$ denotes the slow time scale, that is $\Delta t_{slow} = 0.25$.
The corresponding price dynamics are depicted in Figure \ref{f:slowscale}, which depicts a sequence of square-root like curves. Each one evolves according to the estimated behavior discussed in this  subsection.

\begin{figure}
\begin{center}
\includegraphics[width=0.5\textwidth]{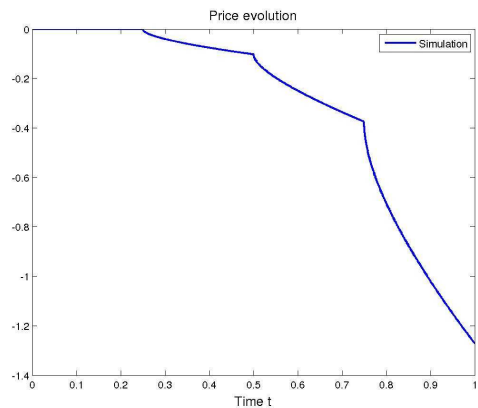}
\caption{Price evolution under 'slow scale' market size fluctuations}\label{f:slowscale}
\end{center}
\end{figure}

\section*{Acknowledgment}
\noindent We are grateful to Panagiotis Souganidis for generously sharing his ideas with us during the course of this research and for helpinng us 
substantially in improving the presentation.\\

\noindent MTW acknowledges financial support from the Austrian Academy of Sciences \"OAW via the New Frontiers Group NFG-001.
PES was partially supported by the NSF.\\

\bibliographystyle{plain}
\bibliography{priceformation}

\end{document}